\documentclass{colt2012} % Anonymized submission

\usepackage{algorithm,algorithmic}
\usepackage{amssymb, multirow, paralist, color}
\newtheorem{thm}{Theorem}
\newtheorem{prop}{Proposition}
\newtheorem{cor}[thm]{Corollary}
\newtheorem{ass}[thm]{Assumption}
\usepackage{enumitem}

\def \R {\mathbb{R}}

\usepackage{scrextend}

 \usepackage{pifont}

\begin{document}

\title[Homotopy Smoothing for Non-Smooth Problems]{Homotopy Smoothing for Non-Smooth Problems\\ with Lower Complexity than $O(1/\epsilon)$ }
\author{\Name{Yi Xu}$^\dag$\Email{yi-xu@uiowa.edu} \\
\Name{Yan Yan}$^\ddag$\thanks{The work of Y. Yan was done when he was a visiting student  with T. Yang at Department of Computer Science of the University of Iowa. }\Email{yan.yan-3@student.uts.edu.au} \\
\Name{Qihang Lin}$^\natural$\Email{qihang-lin@uiowa.edu} \\
\Name{Tianbao Yang}$^\dag$\ding{41} \Email{tianbao-yang@uiowa.edu}\\
       \addr$^\dag$ Department of Computer Science, The University of Iowa, Iowa City, IA 52242, USA \\
       \addr$^\ddag$ QCIS, University of Technology Sydney, NSW 2007, Australia\\ 
       \addr$^\natural$ Department of Management Sciences, The University of Iowa, Iowa City, IA 52242, USA\\
      }
      
\maketitle
\begin{abstract}
In this paper, we develop a novel {\bf ho}moto{\bf p}y  {\bf s}moothing (HOPS) algorithm for solving a family of non-smooth problems that is composed of a non-smooth term with an explicit max-structure and  a smooth term or  a simple non-smooth term whose proximal mapping is easy to compute. The best known iteration complexity for solving such non-smooth optimization problems is $O(1/\epsilon)$ without any assumption on the strong convexity. In this work, we will show that the proposed  HOPS achieved a lower iteration complexity of $\widetilde O(1/\epsilon^{1-\theta})$\footnote{$\widetilde O()$ suppresses a logarithmic factor.} with $\theta\in(0,1]$ capturing the local sharpness of the objective function around the optimal solutions. To the best of our knowledge, this is the lowest iteration complexity achieved so far for the considered non-smooth optimization problems without strong convexity assumption.  The HOPS algorithm employs Nesterov's smoothing technique and Nesterov's accelerated gradient method and runs in stages, which gradually decreases the smoothing parameter in a stage-wise manner until it yields a sufficiently good approximation of the original function. We show that HOPS enjoys a linear convergence for many well-known non-smooth problems (e.g., empirical risk minimization with a piece-wise linear loss function and $\ell_1$ norm regularizer, finding a point in a polyhedron, cone programming, etc). Experimental results verify the effectiveness of HOPS in comparison with Nesterov's smoothing algorithm and the primal-dual style of first-order methods. 
\end{abstract}

\section{Introduction}
In this paper, we consider the following optimization problem:
\begin{align}\label{eqn:opt}
\min_{x\in\Omega_1}F(x)\triangleq f(x) + g(x)
\end{align}
where $g(x)$ is a convex (but not necessarily smooth) function, $\Omega_1$ is a closed convex set and $f(x)$ is a convex but non-smooth function which can be explicitly  written as 
\begin{align}\label{eqn:g}
f(x) = \max_{u\in\Omega_2}\langle Ax, u\rangle  - \phi(u)
\end{align}
where $\Omega_2\subset\R^m$ is a closed convex bounded  set, $A\in\R^{m\times d}$, $\phi(u)$ is a convex function, and $\langle\cdot, \cdot\rangle$ is a scalar product. This family of non-smooth optimization problems  have applications in numerous domains, e.g., machine learning and statistics~\citep{chen2012}, image processing~\citep{Chambolle:2011:FPA:1968993.1969036}, SDP programming~\citep{DBLP:journals/mp/Nesterov07a},  cone programming~\citep{DBLP:journals/mp/LanLM11}, and etc. Several first-order methods have been developed for solving such non-smooth optimization problems including the primal-dual methods~\citep{nemirovski-2005-prox,Chambolle:2011:FPA:1968993.1969036}, Nesterov's smoothing algorithm~\citep{Nesterov:2005:SMN,DBLP:journals/siamjo/Nesterov05}~\footnote{The algorithm in~\citep{Nesterov:2005:SMN} was developed for handling a smooth component  $g(x)$, which can be extended to handling a non-smooth component $g(x)$ whose proximal mapping is easy to compute.}, and they can achieve $O(1/\epsilon)$ iteration complexity for finding an $\epsilon$-optimal solution, which is faster than the corresponding black-box lower complexity bounds by an order of magnitude. 

In this paper, we propose a novel homotopy smoothing (HOPS) algorithm for solving the problem in~(\ref{eqn:opt}) that achieves a lower iteration complexity than $O(1/\epsilon)$. In particular, the iteration complexity of HOPS is given by $\widetilde O(1/\epsilon^{1-\theta})$, where $\theta\in(0,1]$ captures the local sharpness (defined shortly) of the objective function around the optimal solutions. 
The proposed HOPS algorithm builds on the Nesterov's smoothing technique, i.e., approximating the non-smooth function $f(x)$ by a smooth function and optimizing the smoothed function to a desired  accuracy level. 

The striking difference between HOPS and  Nesterov's smoothing algorithm is that Nesterov uses a  fixed small smoothing parameter that renders a sufficiently accurate approximation of the non-smooth function $f(x)$, while HOPS adopts a homotopy strategy for setting the value of the smoothing parameter. It starts from a relatively large smoothing parameter and gradually decreases the smoothing parameter in a stage-wise manner until the smoothing parameter reaches a level that gives  a sufficiently good approximation of the non-smooth objective function.  %For each smoothing parameter, HOPS employs the Nesterov's accelerated optimal methods with warm-start to optimize the smoothed function to a certain accuracy such that the objective  gap to the optimal value is reduced by a certain factor. 
The benefit of using a homotopy strategy is that a  larger smoothing parameter yields a smaller smoothness constant and hence a lower iteration complexity for smoothed problems in earlier stages. For smoothed problems in later stages with larger smoothness constants,  warm-start can help reduce the number of iterations to converge.  As a result, solving a series of smoothed approximations with a smoothing parameter  from large to small and with warm-start is faster than solving one smoothed approximation with a very small smoothing parameter. To the best of our knowledge, {\bf this is the first work} that rigorously analyzes such a homotopy smoothing algorithm and establishes its theoretical guarantee on lower iteration complexities. The keys to our analysis of lower iteration complexity are (i) to leverage a {\bf global error inequality} (Lemma 1)~\citep{DBLP:journals/corr/arXiv:1512.03107} that bounds the distance of a solution to the $\epsilon$ sublevel set by a multiple of the functional distance; and (ii) to explore a local error bound condition to bound the multiplicative factor.

\section{Related Work}
In this section, we review some related work for solving the considered family of  non-smooth optimization problems. Traditional first-order methods such as subgradient descent for solving non-smooth optimization suffer from an $O(1/\epsilon^2)$ iteration complexity. Below, we review some related work for solving~(\ref{eqn:opt}) or its special cases with improved iteration complexities. There are two categories of algorithms, one is based on Neterov's smoothing technique and another one is primal-dual style of first-order methods. 

In the seminal  paper by \citet{Nesterov:2005:SMN}, he proposed a smoothing technique for a family of structured  non-smooth  optimization problems as in~(\ref{eqn:opt}) with $g(x)$ being a smooth function and $f(x)$ given in~(\ref{eqn:g}). By adding a strongly convex prox function in terms of $u$ with a smoothing parameter $\mu$ into the definition of $f(x)$, one can obtain a smoothed approximation of the original objective function. Then he developed an accelerated gradient method with an $O(1/t^2)$ convergence rate for the smoothed objective function with $t$ being the number of iterations, which implies an $O(1/t)$ convergence rate for the original objective function by  setting $\mu \approx  c/t$ with $c$ being a constant. Although he only considered a smooth component $g(x)$ in the original paper,  the algorithm and theory can be easily generalized to handle  a non-smooth component $g(x)$ assuming its proximal mapping is simple to compute by using accelerated proximal gradient methods for composite optimization problems~\citep{Composite,Beck:2009:FIS:1658360.1658364}. Later on, Nesterov proposed an excessive gap technique for solving the similar problem~\citep{DBLP:journals/siamjo/Nesterov05}, which avoids setting the value of the smoothing parameter with the number of iterations or the accuracy  given in advance. In~\citep{DBLP:journals/siamjo/Nesterov05}, Nesterov treats $g(x)$ and $\phi(u)$ symmetrically and simultaneously minimizes the smoothed lower bound and maximizes the smoothed upper bound by updating the primal and dual variables and  iteratively reducing the smoothing  parameters. He established an $O(1/t)$ convergence rate for linear functions $g(x)$ and $\phi(u)$. He also analyzed the case when $g(x)$ is strongly convex, which gives an improved convergence rate of  $O(1/t^2)$. 
The smoothing technique has been exploited to solving problems in machine learning~\citep{citeulike:11703809,DBLP:conf/icml/OuyangG12,DBLP:journals/oms/LinCP14} and statistics~\citep{chen2012}, and cone programming~\citep{DBLP:journals/mp/LanLM11,DBLP:journals/mp/Nesterov07a}.

The primal-dual style of first-order methods treat the problem as a convex-concave minimization problem, i.e., 
\[
\min_{x\in\Omega_1}\max_{u\in\Omega_2}g(x)  + \langle Ax, u\rangle - \phi(u)
\]
\citet{nemirovski-2005-prox} proposed a mirror prox method, which has a convergence rate of $O(1/t)$ by assuming that both $g(x)$ and $\phi(u)$ are smooth functions.~\citet{Chambolle:2011:FPA:1968993.1969036} designed first-order primal-dual algorithms, which tackle  $g(x)$ and $\phi(u)$ using proximal mapping and achieve the same convergence rate of $O(1/t)$ without assuming smoothness of $g(x)$ and $\phi(u)$. When $g(x)$ or $\phi(u)$ is strongly convex, their algorithms achieve $O(1/t^2)$ convergence rate. The effectiveness of  their algorithms was demonstrated on imaging problems. Recently, the primal-dual style of first-order methods have been employed to solve non-smooth optimization problems in machine learning where both the loss function and the regularizer are non-smooth~\citep{DBLP:journals/ml/YangMJZ14Non}.~\citet{DBLP:journals/mp/LanLM11} also considered Nemirovski's prox method for solving cone programming problems. 

The key condition for us to develop an improved  convergence is closely related to local error bounds (LEB)~\citep{DBLP:journals/mp/Pang97} and more generally the Kurdyka-\L ojasiewicz property~\citep{lojasiewicz1965ensembles,Bolte:2006:LIN:1328019.1328299}. The LEB characterizes the relationship between  the distance of a local solution to the optimal set and the optimality  gap of the solution in terms of objective value. The Kurdyka-\L ojasiewicz property characterizes that property of a function that whether it can be made ``sharp''  by some transformation. %The LGC/LEB has been explored for showing improved convergence of gradient descent methods for smooth functions thirty  years ago~\citep{citepulike:6904106}. 
%Recently, these conditions/properties have been  explored for non-smooth optimization~\citep{DBLP:journals/mp/GilpinPS12},  composite optimization~\citep{DBLP:conf/nips/HouZSL13}, gradient and subgradient methods~\citep{arxiv:1510.08234,DBLP:journals/corr/arXiv:1512.03107}. Two mostly  related work are discussed in order. 
Recently, these conditions/properties have been explored for feasible descent methods~\citep{luo1993error}, non-smooth optimization~\citep{DBLP:journals/mp/GilpinPS12}, gradient and subgradient methods~\citep{arxiv:1510.08234,DBLP:journals/corr/arXiv:1512.03107}. {\bf It is notable} that our local error bound condition is different from the one used in~\citep{luo1993error,DBLP:journals/corr/ZhouS15a} which  bounds the distance of a point to the optimal set by the norm of the projected or proximal gradient at that point instead of the functional distance, consequentially it requires some smoothness assumption about the objective function. By contrast, the local error bound condition in this paper covers a much broad family of functions and thus it is more general. Recent work~\citep{DBLP:journals/corr/nesterov16linearnon,HuiZhang16b} have shown that the error bound in~\citep{luo1993error,DBLP:journals/corr/ZhouS15a} is a special case of our considered error bound with $\theta=1/2$. Two mostly related work leveraging a similar error bound to ours are discussed in order. 
\citet{DBLP:journals/mp/GilpinPS12} considered the  two-person zero-sum games, which is a special case of~(\ref{eqn:opt}) with $g(x)$ and $\phi(u)$ being zeros and $\Omega_1$ and $\Omega_2$ being polytopes. %They developed an error bound condition for such problems whose epigraph is a polytope and designed a linearly convergent primal-dual style algorithm. 
The present work is a non-trivial generalization of their work that leads to improved convergence for a much broader family of non-smooth optimization problems. In particular, their result is just a special case of our result when the constant $\theta$ that captures the local sharpness is one for problems whose epigraph is a polytope. Recently,~\citet{DBLP:journals/corr/arXiv:1512.03107} proposed a restarted subgradient method by exploring the local error bound condition or more generally the  Kurdyka-\L ojasiewicz property, resulting in an $\widetilde O(1/\epsilon^{2(1-\theta)})$ iteration complexity with the same constant of $\theta$. In contrast, our result is  an improved iteration complexity of $\widetilde O(1/\epsilon^{1-\theta})$. 

It is worth emphasizing that the proposed homotopy smoothing technique is different from recently proposed homotopy methods for sparse learning (e.g.,  $\ell_1$ regularized  least-squares problem~\citep{DBLP:journals/siamjo/Xiao013}), though a homotopy strategy on an involved parameter is also  employed to boost the convergence. In particular, the involved parameter in the homotopy methods for sparse learning is the regularization parameter before the $\ell_1$ regularization, while the parameter in the present work is the introduced smoothing parameter. In addition, the benefit of starting from a relatively large regularization parameter in sparse learning is  the sparsity of the solution, which makes it possible to explore the restricted strong convexity for proving  faster convergence. We do not make such assumption of the data and we are mostly interested in that when both $f(x)$ and $g(x)$ are non-smooth. %In contrast, the benefit of starting from a relatively large smoothing parameter in HOPS is a smaller smoothness constant of the optimization problem. More importantly, HOPS targets on more general structured non-smooth optimization problems. 

Lastly, we discuss several closely related recent work that also employ homotopy strategies on the smoothing parameter, but are different in how to decrease the smoothing parameter and in the iteration complexities from the proposed HOPS.   \citet{arxiv:1509.00106} proposed an adaptive smoothing algorithm by combining Nesterov's accelerated proximal gradient method and a homotopy strategy for smoothing parameter. Different from Nesterov's smoothing,  their algorithm and analysis require that the added prox function  is not only $\mu$-strongly convex but also smooth with a smoothness parameter $\beta \geq \mu$. In addition, the smoothing parameter is decreased iteratively in the order of $O(1/t)$ where $t$ is the iteration number.  In terms of iteration complexity, when $\beta > \mu$,  their method has an iteration complexity  of $\widetilde O(1/\epsilon)$, and when $\beta = \mu$ it achieves the same  iteration complexity of $O(1/\epsilon )$ as Nesterov's smoothing method \citep{Nesterov:2005:SMN}. In contrast, the proposed HOPS employs a different homotopy strategy that decreases the smoothing parameter geometrically in a stage-wise manner and has a better iteration complexity. 

In~\citep{arXiv:1511.02974}, the authors introduced a new smooth approximation  algorithm by leveraging the strict lower bound of the objective function and a function growth condition. The function growth condition is an inequality  that  the distance of a point to the optimal solution set is less than  a growth constant multiple of the difference between the objective value at the point and the strict lower bound. Their algorithm also has two loops where the outer loop decreases the smoothing parameter  according to the difference between the objective value at current solution and the strict lower bound and the inner loop exploits Nesterov's accelerated gradient method  to solve the intermediate smoothed problem until the relative improvement in terms of the strict lower bound is above $0.8$. In terms of the iteration complexity, their algorithm has an $O\left( \frac{\log H}{ \sqrt{\epsilon'}} + \frac{1}{\epsilon'}\right)$ iteration  complexity for obtaining  an $\epsilon'$-relative optimal solution,  i.e., $F(x_t) - F_* \leq \epsilon' (F_* - F_{slb})$~\footnote{$F_*$ is the optimal objective value and $F_{slb}$ is its strict lower bound}, where $H = \frac{F(x_0) - F_*}{F_* - F_{slb}}$. As discussed in~\citep{arXiv:1511.02974}, their algorithm is favorable when the distance of the initial solution to the optimal set is sufficiently large as their iteration complexity has a logarithmic dependence on the initial solution. However, their algorithm still suffers $O(1/\epsilon)$ iteration complexity in the worst case. In contrast, HOPS leverages the LEB condition instead of the function growth condition so that it not only enjoys a logarithmic dependence on the initial solution but also a reduced iteration complexity. 

More recently, \citet{arXiv:1603.05642} proposed black-box reduction methods for convex optimization by reducing the objective function to a $\beta$-smooth and $\mu$-strongly convex  function and employing an algorithm which satisfies the homotopy objective decrease (HOOD) property (defined shortly). For a non-smooth and  non-strongly convex objective function, they propose to add a $\mu$-strongly convex regularization term $\frac{\mu}{2}\| x - x_0 \|^2$, where $x_0$ is a starting point, and smooth the non-smooth function in the finite sum form $\sum_i f_i(x)$ by applying Nesterov's smoothing technique to  its Fenchel conjugate, which results in a smooth and strongly convex function. By assuming each $f_i(x)$ is $G$-Lipschitz continuous and $\|x_0-x_*\| \leq D$, their method enjoys  an $O(GD/\epsilon)$ iteration complexity by employing accelerated gradient descent method for solving the resulted smooth and strongly convex optimization problems. Their reduction method uses a similar homotopy strategy on the smoothing parameter as HOPS.  Nonetheless, we emphasize that  HOPS is fundamentally different from their method, in particular the inner loop in their method is to ensure the HOOD property of the black-box algorithm  for minimizing  the intermediate smooth and strongly convex function $f(x)$, i.e., for any starting point $x_0$, it produces an output $x'$ satisfying $f(x') - \min f(x) \leq \frac{f(x_0)-\min f(x)}{4}$; in contrast,  the inner loop of HOPS is to solve the intermediate smoothed problem to an accuracy that matches the oder of the current smoothing parameter. By leveraging the local error bound condition, we are able to achieve an iteration complexity with a better dependence on the accuracy level $\epsilon$ and the distance of the initial solution to the optimal set. %Third, the finite-sum form of non-smooth function $f(x)$ that they considered is only one example in our problems. 

Finally, we note that a similar homotopy strategy is employed in Nesterov's smoothing algorithm  for solving an $\ell_1$ norm minimization problem subject to a constraint for recovering a sparse solution~\citep{Becker:2011:NFA:2078698.2078702}. However, we would like to draw readers' attention to that they did not provide any theoretical guarantee on the iteration complexity of the homotopy strategy and consequentially their implementation is ad-hoc without guidance from theory. More importantly, our developed algorithms and theory apply  to a much broader family of problems. 

\section{Preliminaries}
\vspace*{-0.1in}
We present some preliminaries in this section. Let $\|x\|$ denote the Euclidean norm on the primal variable $x$.
A function $h(x)$ is $L$-smooth  in terms of $\|\cdot\|$, if 
\[
\|\nabla h(x) - \nabla h(y)\|\leq L\|x - y\|
\]
We remark here that the generalization to a smoothness definition with respect to a $p$-norm  $\|\cdot\|_p$ with $p\in(1,2]$ is mostly straightforward. We defer the discussion to the supplement. 
%Let $\omega(x; x_0) = \frac{1}{2}\|x - x_0\|^2$, which is assumed to be  $\sigma_1$-strongly convex function w.r.t $x$ in terms of $\|\cdot\|$. In particular, this holds  if we consider $\|x\| = \|x\|_p$ with $p\in(1,2]$ -  the standard  $p$-norm. 
Let $\|u\|_+$ denote a norm on the dual variable, which is not necessarily the Euclidean norm. %Let $\|\cdot\|^*$ denote a dual norm of a norm $\|\cdot\|$. Without loss of generality, 
Denote by $\omega_+(u)$  a 1-strongly convex function of $u$ in terms of $\|\cdot\|_+$. %A function $h(x)$ is $L$-smooth w.r.t $x$ in terms of $\|x\|$, if 
%\[
%\|\nabla h(x) - \nabla h(y)\|^*\leq L\|x - y\|
%\]

For the optimization problem in~(\ref{eqn:opt}), we let $\Omega_*, F_*$ denote the set of optimal solutions and optimal value, respectively, and make the following assumption throughout the paper. 
%\vspace{-0.1in}
\begin{ass}\label{ass:1} For a convex minimization problem~(\ref{eqn:opt}), we assume
(i) there exist $x_0\in\Omega_1$ and $\epsilon_0\geq 0$ such that $F(x_0) - \min_{x\in\Omega_1}F(x)\leq \epsilon_0$;
(ii) $f(x)$ is characterized as in~(\ref{eqn:g}), where $\phi(u)$ is a convex function; 
(iii) There exists a  constant $D$ such that $\max_{u\in\Omega_2}\omega_+(u)\leq D^2/2$; 
(iv) $\Omega_*$ is a non-empty convex compact set.
\end{ass}
Note that: 1) Assumption~\ref{ass:1}(i) assumes that the objective function is lower bounded; 2)  Assumption~\ref{ass:1}(iii) assumes that $\Omega_2$ is a bounded set, which is also required in~\citep{Nesterov:2005:SMN}. 

In addition, for brevity we assume that $g(x)$ is simple enough~\footnote{If $g(x)$ is smooth, this assumption can be relaxed. We will defer the discussion and result on a smooth function $g(x)$ to Section~\ref{smoothHOPS}.}  such that the proximal mapping defined below is easy to compute similar to~\citep{Chambolle:2011:FPA:1968993.1969036}:
\begin{align}\label{prox:g}
P_{\lambda g}(x) = \min_{z\in\Omega_1} \frac{1}{2}\|z - x\|^2 + \lambda g(z)
\end{align}
Relying on the proximal mapping, the key updates in the optimization algorithms presented below take  the following form:
\begin{align}\label{eqn:prox}
\Pi^c_{v, \lambda g}(x) =\arg\min_{z\in\Omega_1}\frac{c}{2}\|z - x\|^2 +  \langle v, z\rangle  + \lambda g(z) 
\end{align}
%Note that if $g$ is smooth and its proximal mapping is easy to compute, we can replace the second form of update with the first form of update, which will lead to faster convergence. We refer to the first form of update as proximal mapping and to the second form of update as gradient mapping. 
For any $x\in\Omega_1$, let $x^*$ denote the closest optimal solution in $\Omega_*$ to $x$ measured in terms of norm $\|\cdot\|$, i.e., $x^* = \arg\min_{z\in\Omega_*}\|z - x\|^2$, 
which is unique because $\Omega_*$ is a non-empty convex compact set~\citep{DBLP:conf/nips/HouZSL13}.  We denote by  $\mathcal L_\epsilon$ the  $\epsilon$-level set of $F(x)$ and  by $\mathcal S_\epsilon$  the $\epsilon$-sublevel set of $F(x)$, respectively, i.e.,
\begin{align*}
\mathcal L_\epsilon &= \{x\in\Omega_1: F(x) = F_* + \epsilon\},\quad\mathcal S_\epsilon = \{x\in\Omega_1: F(x) \leq F_* + \epsilon\}
\end{align*}
It follows from~\citep[Corollary 8.7.1]{rockafellar1970convex} that the sublevel set $\mathcal S_\epsilon$ is bounded for any $\epsilon\geq 0$ and so as the level set $\mathcal L_\epsilon$ due to that $\Omega_*$ is bounded. Define $dist(\mathcal L_\epsilon, \Omega_*)$ to be the maximum distance of points on the level set $\mathcal L_\epsilon$ to the optimal set $\Omega_*$, i.e.,
\begin{align}\label{eqn:keyB}
dist(\mathcal L_\epsilon, \Omega_*)= \max_{x\in\mathcal L_\epsilon}\left[dist(x,\Omega_*)\triangleq\min_{z\in\Omega_*}\|x - z\|\right].
\end{align}
Due to that $\mathcal L_\epsilon$ and $\Omega_*$ are bounded, $dist(\mathcal L_\epsilon, \Omega_*)$ is also bounded. 
%For any $\epsilon>0$, we define
%$$
%\displaystyle \kappa_\epsilon \triangleq \frac{\epsilon}{B_\epsilon}.
%$$
Let $x_{\epsilon}^\dagger$ denote the closest point in the $\epsilon$-sublevel set to $x$, i.e.,
\begin{equation}\label{eqn:xepsilon}
\begin{aligned}
x_\epsilon^{\dagger}&=\arg\min_{z\in\mathcal S_\epsilon}\|z - x\|^2%_2,\quad \text{s.t.}\quad F(z)\leq F_* + \epsilon.
\end{aligned}
\end{equation}
It is easy to show that $x_\epsilon^{\dagger}\in\mathcal L_\epsilon$ when $x\notin\mathcal S_\epsilon$ (using the KKT condition). 
%In the sequel, we let $\|\partial F(\w)\|_2$ be defined as 
%\[
%\|\partial F(\w)\|_2=\inf_{\mathbf v\in\partial F(\w)}  \|\mathbf v\|_2
%\]

\section{Homotopy Smoothing}
In this section, we first describe Nesterov's smoothing technique, and then present the HOPS algorithm and its convergence analysis. Next, we will discuss the local error bound and its application. Finally, we will discuss the HOPS when $g$ is smooth, and also extend the HOPS to general $p$-norm.
\subsection{Nesterov's Smoothing}
We first present the Nesterov's smoothing technique  and accelerated proximal gradient methods for solving the smoothed problem due to that the proposed algorithm builds upon these techniques.  The idea of smoothing is to construct a smooth function  $f_\mu(x)$ that well approximates $f(x)$. Nesterov considered the following function 
\begin{align*}
f_{\mu}(x) = \max_{u\in\Omega_2}\langle Ax, u\rangle  - \phi(u) -\mu\omega_+(u)
\end{align*}
 It was shown in~\citep{Nesterov:2005:SMN} that $f_\mu(x)$ is smooth w.r.t $\|\cdot\|$ and its smoothness parameter is given by $L_\mu = \frac{1}{\mu}\|A\|^2$ where $\|A\|$ is defined by $\|A\|=\max_{\|x\|\leq 1}\max_{\|u\|_+\leq 1}\langle Ax, u\rangle$. Denote by 
\[
u_\mu(x) = \arg\max_{u\in\Omega_2}\langle Ax, u\rangle  - \phi(u) -\mu \omega_+(u)
\]
The gradient of $f_\mu(x)$ is computed by $\nabla f_\mu(x) = A^{\top}u_\mu(x)$. Then 
\begin{align}\label{eqn:approx}
f_\mu(x)\leq f(x)\leq f_\mu(x) + \mu D^2/2
\end{align}
From the inequality above, we can see that when $\mu$ is very small, $f_\mu(x)$ gives a good approximation of $f(x)$. This motivates us to solve the following composite optimization problem
\begin{align*}
\min_{x\in\Omega_1} F_\mu(x)\triangleq  f_\mu(x) + g(x)
\end{align*}
Many works have studied such an optimization problem~\citep{Beck:2009:FIS:1658360.1658364,citeulike:6604666} and the best convergence rate is given by $O(L_\mu/t^2)$, where $t$ is the total number of iterations. We present a variant of accelerated proximal gradient (APG) methods in Algorithm~\ref{alg:0} that works even with  $\|x\|$ replaced with a general norm as long as its square is strongly convex.  We make several remarks about Algorithm~\ref{alg:0}: (i) the variant here is similar to Algorithm 3 in \citep{citeulike:6604666} and the algorithm proposed in~\citep{Nesterov:2005:SMN} except that the prox function $d(x)$ is replaced by $\|x - x_0\|^2/2$ in updating the sequence of $z_k$, which is assumed to be $\sigma_1$-strongly convex w.r.t $\|\cdot\|$; (ii) If $\|\cdot\|$ is simply  the Euclidean norm, a simplified algorithm with only one update in~(\ref{eqn:prox}) can be used (e.g., FISTA~\citep{Beck:2009:FIS:1658360.1658364});  (iii)  if $L_\mu$ is difficult to compute, we can use the backtracking trick (see~\citep{Beck:2009:FIS:1658360.1658364,citeulike:6604666}). 

%The first algorithm is given in~\citep{RePEc:cor:louvco:2007076}, which is an accelerated dual gradient method. We present the detailed steps in Algorithm~\ref{alg:0}. 
\begin{algorithm}[t]
\caption{An Accelerated Proximal Gradient Method: $\text{APG}(x_0, t, L_\mu)$} \label{alg:0}
\begin{algorithmic}[1]
\STATE \textbf{Input}: the  number of iterations $t$, the initial solution $x_0$, and the smoothness constant $L_\mu$
\STATE Let $\theta_0=1$, $V_{-1}=0$, $\Gamma_{-1}=0$, $z_0=x_0$
\STATE Let $\alpha_k$ and $\theta_k$ be two sequences given in Theorem~\ref{thm:APG}. 
\FOR{$k=0,\ldots, t-1$}
    \STATE Compute $y_{k} = (1-\theta_k)x_k + \theta_k z_k$
    \STATE Compute $v_k=\nabla f_\mu(y_k)$, $V_k =V_{k-1} + \frac{v_k}{\alpha_k}$, and $\Gamma_{k}=\Gamma_{k-1} + \frac{1}{\alpha_k}$ 
     \STATE Compute $z_{k+1}=\Pi^{L_\mu/\sigma_1}_{V_k, \Gamma_kg}(x_0)$ and $x_{k+1} = \Pi^{L_\mu}_{v_k, g}(y_k)$
   \ENDFOR
\STATE \textbf{Output}:  $x_t$
\end{algorithmic}
\end{algorithm}
%The second algorithm is given in~\citep{Beck:2009:FIS:1658360.1658364}, which is an accelerated  gradient method (different names appeared in literature, e.g., FISTA). The detailed steps are given in Algorithm~\ref{alg:1}. 
%\begin{algorithm}[t]
%\caption{Accelerated Gradient Method (AGM)} \label{alg:1}
%\begin{algorithmic}[1]
%\STATE \textbf{Input}: the  number of iterations $t$, and the initial solution $ \w_0$
%\STATE Let $\tau_1 =1$ and $\x_1 = \w_0$ 
%\FOR{$k=1,\ldots, t$}
%    \STATE Compute $\w_k = T_g(\x_k)\triangleq \arg\min_{\w\in\R^d}\nabla f_\mu(\x_k)^{\top}(\w - \x_k) + \frac{L_\mu}{2}\|\w - \x_k\|_2^2 + g(\w)$
%    \STATE Compute $\tau_{k+1} = \frac{1+\sqrt{1+4\tau_k^2}}{2}$
%    \STATE Compute $\x_{k+1} = \w_k  + \frac{\tau_k - 1}{\tau_{k+1}}(\w_k - \w_{k-1})$
%   \ENDFOR
%\STATE \textbf{Output}:  $\w_t$
%\end{algorithmic}
%\end{algorithm}

The following theorem states the convergence result for APG. 
\begin{thm}\label{thm:APG}\citep{Nesterov:2005:SMN,citeulike:6604666}
Let $\theta_k  = \frac{2}{k+2}$, $\alpha_k = \frac{2}{k+1}, k\geq 0$ or $\alpha_{k+1} = \theta_{k+1} = \frac{\sqrt{\theta_k^4+4\theta_k^2} - \theta_k^2}{2}, k\geq 0$.  For any $x\in\Omega_1$, we have
\begin{align}
F_\mu(x_t)  - F_\mu(x)\leq \frac{2L_\mu\|x - x_0\|^2 }{t^2}
\end{align}
\end{thm}
Combining  the above convergence result with the relation  in~(\ref{eqn:approx}), we can establish the iteration complexity of Nesterov's smoothing algorithm for solving the original problem~(\ref{eqn:opt}). 
\begin{cor}\label{cor:1}\citep{Nesterov:2005:SMN}
For any $x\in\Omega_1$, we have
\begin{align}\label{eqn:cs}
F(x_t)- F(x)\leq  \mu D^2/2 + \frac{2L_\mu\|x - x_0\|^2}{t^2} 
\end{align}
In particular in order to have $F(x_t)\leq F_* + \epsilon$, it suffices to set $\mu \leq  \frac{\epsilon}{D^2}$ and $t \geq \frac{2D\|A\|\|x_0 - x_*\|}{\epsilon}$, where $x_*$ is an optimal solution to~(\ref{eqn:opt}). 
\end{cor}
%In order to obtain an iteration complexity of $O(1/\epsilon)$, one can set $\mu$ and $t$ such that 
%\begin{align*}
%&\mu D^2/2  = \frac{\epsilon}{2}\Rightarrow \mu = \frac{\epsilon}{D^2}\\
%&\frac{2\|A\|_2^2\|\w_0 - \w_*\|_2^2}{\mu(t+1)^2} = \frac{\epsilon}{2}\Rightarrow t = \frac{2D\|A\|_2\|\w_0 - \w_*\|_2}{\epsilon}
%\end{align*}
%where $\w_*$ denotes an optimal solution.  Although the iteration complexity $O(1/\epsilon)$ seems unbeatable given that both accelerated dual gradient method and accelerated gradient method are optimal for the considered composite optimization, below we will develop a novel algorithm that achieves faster convergence by leveraging the local sharpness (definition given shortly) of the objective function $F(\w)$. 

\subsection{Homotopy Smoothing}
From the convergence result in~(\ref{eqn:cs}), we can see that in order to obtain a very accurate solution, we have to set $\mu$ - the smoothing parameter - to be a very small value, which will cause the blow-up of the second term because $L_\mu\propto 1/\mu$. On the other hand, if $\mu$ is set to be a relatively large value, then $t$ can be set to be a relatively small value to match the first term, which may lead to a not sufficiently accurate solution. It seems that the $O(1/\epsilon)$ is unbeatable. However, if we adopt a homotopy strategy, i.e., starting from a relatively large value $\mu$ and optimizing the smoothed function with a certain number of iterations $t$ such that the second term in~(\ref{eqn:cs}) matches the first term, which will give $F(x_t) - F(x_*)\leq O(\mu)$. Then we can reduce the value of $\mu$ by a constant factor $b>1$ and warm-start the optimization process again  from $x_t$. The key observation is that although $\mu$ decreases and $L_\mu$ increases, the other term $\|x_* - x_t\|$ is also reduced compared to $\|x_* - x_0\|$, which could  cancel the blow-up effect caused by increased  $L_\mu$. As a result, we expect to use the same number of iterations to optimize the smoothed function with a smaller $\mu$ such that $F(x_{2t}) - F(x_*)\leq O(\mu/b)$. 

To formalize our observation, we first present the following key lemma  below. 
\begin{lemma}[\cite{DBLP:journals/corr/arXiv:1512.03107}]\label{lem:1}
For any $x\in\Omega_1$ and $\epsilon>0$, we have
\[
\|x - x^\dagger_\epsilon\|\leq \frac{dist(x^\dagger_\epsilon, \Omega_*)}{\epsilon}(F(x) - F(x_\epsilon^\dagger))
\]
where $x^\dagger_\epsilon\in\mathcal S_\epsilon$ is the closest point in the $\epsilon$-sublevel set to $x$ as defined in~(\ref{eqn:xepsilon}). 
\end{lemma}

The lemma is proved in~\citep{DBLP:journals/corr/arXiv:1512.03107}. We include its proof in Appendix. If we apply the above bound into~(\ref{eqn:cs}), we will see in the proof of the main theorem (Theorem~\ref{thm:GDr}) that the number of iterations $t$ for solving each smoothed problem   is roughly $O(\frac{dist(\mathcal L_\epsilon, \Omega_*)}{\epsilon})$, which will be lower than $O(\frac{1}{\epsilon})$  in light of the local error bound condition given below. 
\begin{definition}[Local error bound (LEB)]
A function $F(x)$ is said to satisfy a local error bound condition if there exist $\theta\in(0,1]$ and $c>0$ such that for any $x\in\mathcal S_\epsilon$
\begin{align}\label{eqn:leb}
dist(x, \Omega_*)\leq c(F(x) - F_*)^{\theta}
\end{align}
\end{definition}
{\bf Remark:}  In next subsection, we will discuss the relationship with other types of conditions and show that a broad family of non-smooth functions (including almost all commonly seen functions in machine learning) obey the local error bound condition. The exponent constant $\theta$ can be considered as a local sharpness measure of the function. Figure~\ref{fig2} illustrates the sharpness of $F(x) = |x|^p$ for $p=1, 1.5$, and  $2$ around the optimal solutions and their corresponding $\theta$.

With the local error bound condition, we can see that $dist(\mathcal L_\epsilon, \Omega_*)\leq c\epsilon^{\theta}, \theta\in(0,1]$. Now, we are ready to present the homotopy smoothing algorithm and its convergence guarantee under the local error bound condition.  The HOPS algorithm is presented in Algorithm~\ref{alg:2}, which starts from a relatively large smoothing parameter $\mu=\mu_1$ and gradually reduces $\mu$ by a factor of $b>1$ after running  a number $t$ of iterations of APG with warm-start. The iteration complexity of HOPS is established below. 

\begin{algorithm}[t]
\caption{Homotopy Smoothing (HOPS) for solving~(\ref{eqn:opt}) } \label{alg:2}
\begin{algorithmic}[1]
\STATE \textbf{Input}: the number of stages $m$ and  the  number of iterations $t$ per-stage, and the initial solution $x_0\in\Omega_1$ and a parameter $b>1$. 
\STATE Let $\mu_1 = \epsilon_0/(b D^2)$
\FOR{$s=1,\ldots, m$}
    \STATE Let $x_{s} = \text{APG}(x_{s-1}, t, L_{\mu_{s}})$ 
    \STATE Update $\mu_{s+1} = \mu_s/b$ 
   \ENDFOR
\STATE \textbf{Output}:  $x_m$
\end{algorithmic}
\end{algorithm}

\begin{figure*}[t]
%\hspace*{0.1in}
%\begin{minipage}[t]{0.53\textwidth}\vspace*{0.1in}
\centering
\includegraphics[scale=0.35]{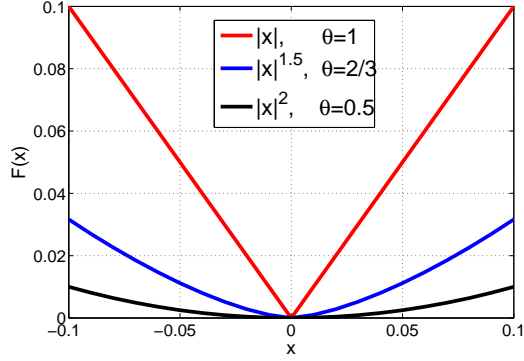}
%\vspace*{-0.12in}
\caption{Illustration of local sharpness of three  functions and the corresponding $\theta$ in the LEB condition.}
\label{fig2}
%\end{minipage}
%\vspace*{-0.15in}
\end{figure*}

%\begin{figure*}[t]
%\hspace*{-0.12in} \begin{minipage}[t]{0.46\textwidth}
%%\begin{algorithm}[t]
%\begin{algorithm}[H]
%\caption{HOPS for solving~(\ref{eqn:opt}) } \label{alg:2}
%\begin{algorithmic}[1]
%\STATE \textbf{Input}: $m$, $t$,  $x_0\in\Omega_1,\epsilon_0, D^2$ and $b>1$. 
%\STATE Let $\mu_1 = \epsilon_0/(b D^2)$
%\FOR{$s=1,\ldots, m$}
%    \STATE Let $x_{s} = \text{APG}(x_{s-1}, t, L_{\mu_{s}})$ 
 %   \STATE Update $\mu_{s+1} = \mu_s/b$ 
  % \ENDFOR
%\STATE \textbf{Output}:  $x_m$
%\end{algorithmic}
%\end{algorithm}
%\end{minipage}
%\hspace*{0.1in}
%\begin{minipage}[t]{0.53\textwidth}\vspace*{0.1in}
%\centering
%\includegraphics[scale=0.2]{sharp-2.eps}
%\vspace*{-0.12in}
%\caption{Illustration of local sharpness of three  functions and the corresponding $\theta$ in the LEB condition.}
%\end{minipage}
%\vspace*{-0.15in}
%\end{figure*}

\begin{thm}\label{thm:GDr}
Suppose Assumption~\ref{ass:1} holds and $F(x)$ obeys the local error bound condition.  Let HOPS run with $t=O(\frac{2bcD\|A\|}{\epsilon^{1-\theta}})\geq \frac{2bcD\|A\|}{\epsilon^{1-\theta}}$ iterations for each stage, and $m = \lceil \log_b(\frac{\epsilon_0}{\epsilon})\rceil$. Then  %there exists $s\in\{1,\ldots, m\}$  %or after at most  $K$ epochs, we have 
\begin{align*}
F(x_m) - F_*\leq 2\epsilon. 
\end{align*}
%depending on using either subgradients or stochastic subgradients. 
Hence, the iteration complexity for achieving an $2\epsilon$-optimal solution is $\frac{2bcD\|A\|}{\epsilon^{1-\theta}}\lceil \log_b(\frac{\epsilon_0}{\epsilon})\rceil$ in the worst-case.
\end{thm}

\begin{proof}
%We assume that the epoch solutions $x_s \notin \S_{2\epsilon}$ for $s = 1, 2, \dots, m-1$ and induce how many epochs are sufficient for reaching to a point in $2\epsilon$-sublevel set in the worst case. Let $x_{s,\epsilon}^\dag$ denote the closest point to $x_s$ in the $\epsilon$ sublevel set. Since $\S_{\epsilon}\subseteq \S_{2\epsilon}$, therefore $x_s \not\in \S_\epsilon$ and consequentially  $F(x_{s,\epsilon}^\dagger)=F_* + \epsilon$ (using the KKT condition) for $s=1,\ldots, m-1$.
Let $x_{s,\epsilon}^\dag$ denote the closest point to $x_s$ in the $\epsilon$ sublevel set.  Define  $\epsilon_s \triangleq\frac{\epsilon_0}{b^s}$. Note that $\mu_s = \epsilon_s/D^2$. We will show by induction that $F(x_s) - F_*\leq \epsilon_s +\epsilon$ for $s=0,1,\dots$ which leads to our conclusion when $s=m$. 
The inequality holds obviously for $s=0$. Assuming $F(x_{s-1}) - F_*\leq \epsilon_{s-1} + \epsilon$, we  need to show that $F(x_s) - F_*\leq \epsilon_s +\epsilon$. We apply Corollary \ref{cor:1} to the $s$-th epoch of Algorithm \ref{alg:2} and get
\begin{align}\label{eqn:th5-1}
F(x_s) - F(x_{s-1, \epsilon}^\dagger) \leq \frac{D^2\mu_s}{2} + \frac{2\|A\|^2\|x_{s-1} - x_{s-1,\epsilon}^\dagger\|^2}{\mu_s t^2}
\end{align}
First, we assume $F(x_{s-1}) - F_* \leq \epsilon$, i.e. $x_{s-1} \in \mathcal S_{\epsilon}$. Then we have $x_{s-1, \epsilon}^\dagger = x_{s-1}$ and 
\begin{align*}
F(x_s) - F(x_{s-1, \epsilon}^\dagger) \leq \frac{D^2\mu_s}{2} \leq \frac{\epsilon_s}{2}
\end{align*}
As a result, 
\begin{align*}
F(x_s) - F_* \leq F(x_{s-1, \epsilon}^\dagger) - F_* + \frac{\epsilon_s}{2} \leq \epsilon + \epsilon_s
\end{align*}
Next, we consider $F(x_{s-1}) - F_* > \epsilon$, i.e. $x_{s-1} \notin \S_{\epsilon}$. Then we have $F(x_{s-1, \epsilon}^\dagger)  - F_* = \epsilon$.
By Lemma 1, we have
\begin{align}\label{eqn:th5-2}
\nonumber \|x_{s-1} - x^\dagger_{s-1,\epsilon}\| &\leq \frac{dist(x^\dagger_{s-1,\epsilon}, \Omega_*)}{\epsilon} (F(x_{s-1}) - F(x_{s-1,\epsilon}^\dagger))\\
\nonumber & \leq \frac{dist(x^\dagger_{s-1,\epsilon}, \Omega_*)}{\epsilon} [\epsilon_{s-1}+\epsilon - \epsilon] = \frac{dist(x^\dagger_{s-1,\epsilon}, \Omega_*)\epsilon_{s-1}}{\epsilon} \\
\nonumber & \leq \frac{c(F(x^\dagger_{s-1,\epsilon}) - F_*)^\theta \epsilon_{s-1}}{\epsilon} \\
& \leq \frac{c(\epsilon)^\theta \epsilon_{s-1}}{\epsilon} =  \frac{c \epsilon_{s-1}}{\epsilon^{1-\theta}}
\end{align}
%where the second inequality is from the fact that
%%\[
%%f(\w_{k-1}) - f(\w^\dagger_{k-1,\epsilon})\leq f(\w_{k-1}) - f_*  \leq \epsilon_{k-1} + (a_{k-1}-1)\epsilon.
%%\]
%\begin{align*}
%&\E[f(\w_{k-1}) - f(\w^\dagger_{k-1,\epsilon})]\\
%&\leq \E[f(\w_{k-1}) - f_*  + (f_* - f(\w^\dagger_{k-1,\epsilon}))]\\
%&\leq \epsilon_{k-1} + \epsilon -\epsilon = \epsilon_{k-1} 
%\end{align*}
Combining (\ref{eqn:th5-1}) and (\ref{eqn:th5-2}) and using the fact that $\displaystyle \mu_s = \frac{\epsilon_s}{D^2}$  and $\displaystyle t \geq \frac{2bc D\|A\|}{\epsilon^{1-\theta}}$, we have 
\begin{align*}
&F(x_s) - F(x_{s-1, \epsilon}^\dagger) \leq   \frac{\epsilon_s}{2} + \frac{\epsilon_{s-1}^2}{2\epsilon_s b^2} = \epsilon_s 
%&=\frac{\epsilon_k}{2} +  \frac{\epsilon_{k-1} }{2\alpha} = \epsilon_k ,
\end{align*}
which together with the fact that $ F(x^\dagger_{s-1,\epsilon})=F_*+\epsilon$ implies
\begin{align*}
F(x_s) - F_*& \leq \epsilon  + \epsilon_s%\\
% &= \epsilon_k  + \frac{(a_{k-1}-1)\epsilon+\alpha\epsilon}{\alpha} =\epsilon_k + a_k\epsilon
\end{align*}
Therefore by induction, we have
\[
F(x_m) - F_*\leq \epsilon_m + \epsilon = \frac{\epsilon_0}{b^m} +\epsilon\leq 2\epsilon
\]
where the last inequality is due to the value of $m$. 

%{\bf Remark:} The theorem shows that if the objective function $F(x)$ satisfies  the local error bound which defined in \textbf{Defination 4}. Then, HOPS can achieve a lower iteration complexity than $\tilde O(1/\epsilon^{1-\theta})$ with $\theta \in (0,1]$. 

\end{proof}

%Proof Sketch: Assume that if  $F(x_s)-F_*>2\epsilon$  for any $s\in\{1,\ldots, m-1\}$, then we prove  $F(x_m) - F_*\leq 2\epsilon$  by induction. In particular, we assume $F(x_{s-1}) - F_*\leq\epsilon_{s-1}+\epsilon$, where $\epsilon_s = \epsilon_0/b^s$, and prove that $F(x_s) - F_*\leq \epsilon_s + \epsilon$. We apply Corollary~\ref{cor:1} to the $s$-th epoch of HOPS and get 
%\begin{align*}
%F(x_s)- F(x_{s-1,\epsilon}^\dagger)\leq  \mu_s D^2/2 + \frac{2\|A\|^2\|x_{s-1,\epsilon}^\dagger - x_{s-1}\|^2}{\mu_s t^2} 
%\end{align*}
%Then by exploring Lemma~\ref{lem:1}, the local error bound condition and the values of $\mu_s$ and $t$, we end up with 
%\[
%F(x_s)  - F(x_{s-1,\epsilon}^\dagger)\leq \epsilon_s\Rightarrow F(x_s) - F_* \leq \epsilon_s + \epsilon
%\]
%The detailed proof can be found in the supplement. 

\subsection{Local Error Bounds and Applications}
In this subsection, we discuss the local error bound condition and its application in non-smooth optimization problems. 
\paragraph{The Hoffman's bound and finding a point in a polyhedron.}
A polyhedron can be expressed as $\mathcal P =\{x\in\R^d; B_1x\leq b_1, B_2x=b_2\}$. The Hoffman's bound~\cite{DBLP:journals/mp/Pang97} is expressed as 
\begin{align}\label{eqn:hoff}
dist(x, \mathcal P) \leq c(\|(B_1x - b_1)_+\| + \|B_2x - b_2\|), \exists c>0
\end{align}
where $[s]_+=\max(0,s)$. This can be considered as the error bound for the polyhedron feasibility problem, i.e., finding a $x\in\mathcal P$, which is equivalent to 
\begin{align*}
\min_{x\in\R^d}F(x)\triangleq \left[\|(B_1x - b_1)_+\| + \|B_2x - b_2\| = \max_{u \in \Omega_2} \langle B_1x-b_1, u_1\rangle +  \langle B_2x-b_2, u_2\rangle \right]
\end{align*}
where $u = (u_1^\top, u_2^\top)^\top$ and $\Omega_2 = \{ u | u_1 \succeq 0, \|u_1\|\leq 1, \|u_2\|\leq 1 \}$. If there exists a $x\in\mathcal P$, then $F_*=0$. Thus the Hoffman's bound in~(\ref{eqn:hoff}) implies a local error bound~(\ref{eqn:leb}) with $\theta=1$. Therefore, the HOPS has a linear convergence for finding a feasible solution in a polyhedron. 
If we let $\omega_+(u) = \frac{1}{2}\|u\|^2$ then $D^2 = 2$ so that the iteration complexity is $2\sqrt{2}bc \max(\|B_1\|,\|B_2\|)\lceil \log_b(\frac{\epsilon_0}{\epsilon})\rceil$.

\paragraph{Cone programming.} Let $U, V$ denote two vector spaces. Given a linear opearator $\mathcal E: U\rightarrow V^*$~\footnote[4]{$V^*$ represents the dual space of $V$. The notations and descriptions are adopted from~\citep{DBLP:journals/mp/LanLM11}.}, a closed convex set $\Omega \subseteq U$, and  a vector $e\in V^*$, and a closed convex cone $\mathcal K\subseteq V$, the general constrained cone linear system (cone programing) consists of finding a vector $x\in\Omega$ such that $\mathcal Ex - e \in\mathcal K^*$. Lan et al.~\cite{DBLP:journals/mp/LanLM11} have considered  Nesterov's smoothing algorithm for solving the cone programming problem with $O(1/\epsilon)$ iteration complexity. The problem can be cast into a non-smooth optimization problem: 
\begin{align*}
\min_{x\in\Omega} F(x)\triangleq \left[dist(\mathcal Ex - e, \mathcal K^*) = \max_{\|u\|\leq 1, u\in -\mathcal K}\langle \mathcal Ex -e, u\rangle\right]
\end{align*}
Assume that $e\in Range(\mathcal E) -\mathcal K^*$, then $F_*=0$.  Burke et al.~\cite{DBLP:journals/siamjo/BurkeT96} have considered the error bound for such problems and their results imply that there exists $c>0$ such that $dist(x,\Omega_*)\leq c (F(x) - F_*)$ as long as $\exists x\in\Omega, \text{s.t. } \mathcal Ex - e \in int(\mathcal K^*)$, where $\Omega_*$ denotes the optimal solution set. Therefore, the HOPS also has a linear convergence for cone programming. 
{Considering that both $U$ and $V$ are Euclidean spaces, we set $\omega_+(u) = \frac{1}{2}\|u\|^2$ then $D^2 = 1$. Thus, the iteraction complexity of HOPS for finding an $2\epsilon$-solution is $2bc \|\mathcal E\|\lceil \log_b(\frac{\epsilon_0}{\epsilon})\rceil$.}

\paragraph{Non-smooth regularized empirical loss (REL) minimization in Machine Learning}
The REL consists of a sum of loss functions on the training data and a regularizer, i.e., 
\begin{align*}
\min_{x\in\R^d}F(x)\triangleq \frac{1}{n}\sum_{i=1}^n\ell(x^{\top}a_i, y_i) + \lambda g(x)
\end{align*}
where $(a_i, y_i), i=1,\ldots, n$ denote pairs of a feature vector and a label of  training data. Non-smooth loss functions include hinge loss $\ell(z,y)=\max(0,1-yz)$, absolute loss $\ell(z,y)=|z-y|$, which can be written as the max structure in~(\ref{eqn:g}). Non-smooth regularizers include e.g., $g(x) = \|x\|_1$, $g(x)=\|x\|_\infty$. These loss functions and regularizers are essentially piecewise linear functions, whose epigraph is a polyhedron. The error bound condition has been developed for such kind of problems~\citep{DBLP:journals/corr/arXiv:1512.03107}. In particular, if $F(x)$ has a  polyhedral epigraph, then there exists $c>0$ such that $dist(x, \Omega_*)\leq c (F(x) - F_*)$ for any $x\in\R^d$. It then implies  HOPS has an $O(\log(\epsilon_0/\epsilon))$ iteration complexity for solving a non-smooth REL minimization with a polyhedral epigraph. \cite{DBLP:journals/ml/YangMJZ14Non} has also considered such non-smooth problems, but they only have $O(1/\epsilon)$ iteration complexity.

%\paragraph{When $F(x)$ is locally strongly convex}in terms of $\|\cdot\|$  such that~\footnote{This is true if $g(x)$ is strongly convex or locally strongly convex.} 
%\begin{align}\label{eqn:st}
%\|x - x_*\|^2\leq \frac{2}{\sigma}(F(x) - F_*), \forall x\in\mathcal S_\epsilon
%\end{align}
%where $\sigma$ is called the strong convexity parameter, then we can see that the local error bound holds with $\theta={1/2}$, which implies the iteration complexity of HOPS is $\tilde O(\frac{1}{\sqrt{\epsilon}})$, which is up to a logarithmic factor the same as the result in~\citep{Chambolle:2011:FPA:1968993.1969036,DBLP:journals/siamjo/Nesterov05} for a strongly convex function. However, here only local strong convexity is sufficient and there is no need to develop  a different algorithm and different analysis from the non-strongly convex case as done in~\citep{Chambolle:2011:FPA:1968993.1969036,DBLP:journals/siamjo/Nesterov05}. %For example, one can consider $F(x) = \frac{1}{2}\|Ax\|_p^2, p\in(1,2]$, which is a composition of  a strongly convex function $\|\cdot\|^2_p$ and an affine function $Ax$, hence satisfying~(\ref{eqn:st}) according to~\citep{DBLP:journals/corr/GongY14}.  
%
\paragraph{When $F(x)$ is essentially locally strongly convex~\citep{goebel2008local}}in terms of $\|\cdot\|$  such that~\footnote[5]{This is true if $g(x)$ is strongly convex or locally strongly convex.} 
\begin{align}\label{eqn:st}
dist^2(x, \Omega_*)\leq \frac{2}{\sigma}(F(x) - F_*), \forall x\in\mathcal S_\epsilon
\end{align}
then we can see that the local error bound holds with $\theta={1/2}$, which implies the iteration complexity of HOPS is $\widetilde O(\frac{1}{\sqrt{\epsilon}})$, which is up to a logarithmic factor the same as the result in~\citep{Chambolle:2011:FPA:1968993.1969036} for a strongly convex function. However, here only local strong convexity is sufficient and there is no need to develop  a different algorithm and different analysis from the non-strongly convex case as done in~\citep{Chambolle:2011:FPA:1968993.1969036}. %For example, one can consider $F(x) = \frac{1}{2}\|Ax\|_p^2, p\in(1,2]$, which is a composition of  a strongly convex function $\|\cdot\|^2_p$ and an affine function $Ax$, hence satisfying~(\ref{eqn:st}) according to~\cite{DBLP:journals/corr/GongY14}. 
For example, one can consider $F(x) = \|Ax-y\|_p^p = \sum_{i=1}^{n}| a_i^\top x - y_i|^p, p\in(1,2)$, which satisfies~(\ref{eqn:st}) according to \citep{DBLP:journals/corr/arXiv:1512.03107}.

\paragraph{The Kurdyka-\L ojasiewicz (KL) property.}
The definition of KL property is given below. 
 \begin{definition}
 Let's define the subdifferential of $F$ at $x$ as $\partial F(x) = \{ u: F(y) \geq F(x) + \langle u, y-x\rangle \text{ for all } y \in \Omega_1\}$. The function $F(x)$ is said to have the KL property at $x_*\in\Omega_*$ if there exist $\eta\in(0,\infty]$, a neighborhood $U$ of $x_*$ and a continuous concave function $\varphi:[0, \eta) \rightarrow \R_+$ such that i) $\varphi(0) = 0$,  $\varphi$ is continuous on $(0, \eta)$, ii) for all $s\in(0, \eta)$, $\varphi'(s)>0$, iii) and for all $x\in U\cup \{x: F(x_*)< F(x)< F(x_*)+\eta\}$, the KL inequality 
 $\varphi'(F(x) - F(x_*))\|\partial F(x)\|\geq 1$ holds. 
 \end{definition}
 The function $\varphi$ is called the desingularizing function of $F$ at $x_*$, which makes  the function $F(x)$  sharp by reparameterization. An important desingularizing function is in the form of $\varphi(s) = cs^{1-\beta}$ for some $c>0$ and $\beta\in[0,1)$, which gives the KL inequality $\|\partial F(x)\|\geq  \frac{1}{c(1-\beta)}(F(x) - F(x_*))^{\beta}$.  It has been established  that the KL property  is satisfied by a wide class of non-smooth functions  definable in an o-minimal structure~\citep{Bolte:2006:LIN:1328019.1328299}. Semialgebraic functions and (globally) subanalytic functions are for instance definable in their respective classes. While the definition of KL property involves a neighborhood $U$ and a constant  $\eta$, in practice  many convex functions satisfy the above property with $U =\R^d$ and $\eta=\infty$~\citep{Attouch:2010:PAM:1836121.1836131}. The proposition below shows that a function with the KL property with a desingularizing function $\varphi(s)=cs^{1-\beta}$ obeys the local error bound condition in~(\ref{eqn:leb}) with $\theta = 1-\beta\in(0,1]$, which implies an iteration complexity of $\widetilde O(1/\epsilon^{\theta})$ of HOPS for optimizing such a function. 
\begin{prop}\label{prop:KL}\citep[Theorem 5]{arxiv:1510.08234}
Let $F(x)$ be a proper, convex and lower-semicontinuous function that satisfies KL property at $x_*$ and $U$ be a neighborhood of $x_*$. For all $x\in U\cap \{x: F(x_*)< F(x)< F(x_*)+\eta\}$, if $\|\partial F(x)\|\geq  \frac{1}{c(1-\beta)}(F(x) - F(x_*))^{\beta}$, then $dist(x, \Omega_*)\leq c(F(x) - F_*)^{1-\beta}$. %Assume that $S_\epsilon \subset \bigcup_{x_* \in \Omega_* } U_{x_*}$, where $U_{x_*}$ is the eighborhood $U$ of a particular $x_*$. 
\end{prop} 
{\bf Remark:} In order to apply the KL property to our method, we usually need to assume the KL property is satisfied at every $x_*$ with $U$ containing $\mathcal S_\epsilon$ in Proposition \ref{prop:KL}, i.e. assume $\mathcal S_\epsilon \subset \bigcup_{x_* \in \Omega_* } U_{x_*}$, where $U_{x_*}$ is the neighborhood $U$ of a particular $x_*$. However, as we mentioned above, in practice many convex functions satisfy the KL property with $U =\R^d$ and $\eta=\infty$~\citep{Attouch:2010:PAM:1836121.1836131} so that above assumption holds.

\subsection{HOPS for a smooth $g(x)$} \label{smoothHOPS}
In the preliminaries section, we assume that $g(z)$ is simple enough such that the proximal mapping defined below is easy to compute:
\begin{align}\label{prox:g}
P_{\lambda g}(x) = \min_{z\in\Omega_1} \frac{1}{2}\|z - x\|^2 + \lambda g(z)
\end{align}
%Relying on the proximal mapping, the key updates in the optimization algorithms presented below take  the following form:
%\begin{align}\label{eqn:prox}
%\Pi^c_{v, \lambda g}(x) =\arg\min_{z\in\Omega_1} \langle v, z\rangle  + \lambda g(z) + \frac{c}{2}\|z - x\|^2
%\end{align}
We claimed that if $g(z)$ is smooth, this assumption can be relaxed. In this section, we present the discussion and result for a smooth function $g(x)$ without assuming that its proximal mapping is easy to compute.  In particular, we will consider $ g$ as a smooth component in $f_\mu + g$ and use the gradient of both $f_\mu$ and $g$ in the updating.  The detailed updates are presented in Algorithm~\ref{alg:APG-smooth}, where
%\begin{align*}
%%\widetilde{\Pi}^c_{v, u}(x) =\arg\min_{z\in\Omega_1} \langle v, z\rangle  + \langle u, z\rangle + \frac{c}{2}\|z - x\|^2
%%\widetilde{\Pi}^c_{v, \lambda g}(x) =
%\arg\min_{z\in\Omega_1} \langle v, z\rangle + \lambda \langle \nabla g(y), z\rangle+ \frac{c}{2}\|z - x\|^2
%\end{align*}
%To simplify, we can write the updates as
\begin{align}\label{eqn:prox_smooth}
\widetilde{\Pi}^c_{u}(x) = \arg\min_{z\in\Omega_1} \langle u, z\rangle + \frac{c}{2}\|z - x\|^2
\end{align}

\begin{algorithm}[t]
\caption{An Accelerated Proximal Gradient Method ($g$ is smooth): $\text{APG}(x_0, t, L_\mu)$} \label{alg:APG-smooth}
\begin{algorithmic}[1]
\STATE \textbf{Input}: the  number of iterations $t$, the initial solution $x_0$, and the smoothness constant $L_\mu$
\STATE Let $\theta_0=1$, $U_{-1}=0$, $z_0 = x_0$%, $\Gamma_{-1}=0$ 
\STATE Let $\alpha_k$ and $\theta_k$ be two sequences given in Theorem 2. %~\ref{thm:APG}. 
\FOR{$k=0,\ldots, t-1$}
    \STATE Compute $y_{k} = (1-\theta_k)x_k + \theta_k z_k$
    \STATE Compute $u_k=\nabla f_\mu(y_k) + \nabla g(y_k)$, $U_k =U_{k-1} + \frac{u_k}{\alpha_k}$% $w_k=\nabla g(y_k)$, $V_k =V_{k-1} + \frac{v_k}{\alpha_k}$, and $\Gamma_{k}=\Gamma_{k-1} + \frac{w_k}{\alpha_k}$ 
     \STATE Compute $z_{k+1}=\widetilde{\Pi}^{(L_\mu+M)/\sigma_1}_{U_k}(x_0)$ and $x_{k+1} = \widetilde{\Pi}^{L_\mu+M}_{u_k}(y_k)$
   \ENDFOR
\STATE \textbf{Output}:  $x_t$
\end{algorithmic}
\end{algorithm}

To present the convergence guarantee, we assume that the function $g$ is $M$-smooth w.r.t $\|x\|$, then the smoothness parameter of objective function $F_\mu(x) = f_\mu(x) + g(x)$ is
\begin{align}
L = L_\mu + M = \frac{\| A \|^2}{\mu} + M
\end{align}
Then, we state the convergence result of Algorithm~\ref{alg:APG-smooth} in the following corollary.
\begin{cor}\label{cor:smoothAPG}
Let $\theta_k  = \frac{2}{k+2}$, $\alpha_k = \frac{2}{k+1}, k\geq 0$ or $\alpha_{k+1} = \theta_{k+1} = \frac{\sqrt{\theta_k^4+4\theta_k^2} - \theta_k^2}{2}, k\geq 0$.  For any $x\in\Omega_1$, we have
\begin{align}
F(x_t)- F(x)\leq  \frac{\mu D^2}{2} + \frac{2\| A \|^2\|x - x_0\|^2}{\mu t^2} + \frac{2 M \|x - x_0\|^2}{t^2}
\end{align}
\end{cor}

{\bf Remark:} In order to have $F(x_t)\leq F(x_*) + \epsilon$,  we can consider $x = x_*$ in Corollary~\ref{cor:smoothAPG}, i.e.
\begin{align}\label{eq:APGupbound}
F(x_t)- F(x_*)\leq  \frac{\mu D^2}{2} + \frac{2\| A \|^2\|x_* - x_0\|^2}{\mu t^2} + \frac{2 M \|x_* - x_0\|^2}{t^2}
\end{align}
In particular, we set 
\begin{align*}
\mu = \frac{2\epsilon}{3D^2} 
\end{align*}
and
\begin{align*}
t \geq \max\left\{ \frac{3D \|A\| \|x_* -x_0\|}{\epsilon}, \frac{\sqrt{6M}\|x_* -x_0\|}{\sqrt{\epsilon}}\right\}
\end{align*}
Algorithm \ref{alg:APG-smooth} also achieves the iteration complecity of $O(1/\epsilon)$.
%Minimizing the right-hand side of above inequality (\ref{eq:APGupbound}) in $\mu$, we have $ $
%\begin{align*}
%\hat{\mu} = \frac{2\| A \| \|x_* - x_0\|}{Dt} 
%\end{align*}
%and
%\begin{align}
%F(x_t)- F(x_*)\leq  \frac{2 D \| A \|\|x_* - x_0\|}{t} + \frac{2 M \|x_* - x_0\|^2}{t^2}
%\end{align}
%Thus, to find an $\epsilon$-solution to the problem (1), we want
%\begin{align}
%t \geq  \frac{D \| A \|+ \sqrt{D^2 \| A \|^2+ 2M\epsilon} }{\epsilon} \|x_* - x_0\|
%\end{align}
%In particular, we set 
%\begin{align}
%\hat{\mu} = \frac{2\| A \| \epsilon}{D^2 \|A\| + D\sqrt{D^2\|A\|^2+2M\epsilon}} 
%\end{align}

Similarly, we can develop the HOPS algorithm and present it in Algorithm~\ref{alg:HOPS-smooth}. The iteration complexity of HOPS is established in Theorem \ref{thm:HOPS-smooth}.
\begin{algorithm}[t]
\caption{Homotopy Smoothing (HOPS) for solving (1) ($g$ is smooth)} \label{alg:HOPS-smooth}
\begin{algorithmic}[1]
\STATE \textbf{Input}: the number of stages $m$ and  the  number of iterations $t$ per-stage, and the initial solution $x_0\in\Omega_1$ and a parameter $b>1$. 
\STATE Let $\mu_1 = \frac{2 \epsilon_0}{3bD^2}$
\FOR{$s=1,\ldots, m$}
    \STATE Let $x_{s} = \text{APG}(x_{s-1}, t, L_{\mu_{s}})$ 
    \STATE Update $\mu_{s+1} = \mu_s/b$ 
   \ENDFOR
\STATE \textbf{Output}:  $x_m$
\end{algorithmic}
\end{algorithm}

\begin{thm}\label{thm:HOPS-smooth}
Suppose Assumption 1 holds and $F(x)$ obeys the local error bound condition.  Let HOPS run with $t =O(1/\epsilon^{1-\theta})\geq \max\left\{ \frac{3D \|A\| bc}{\epsilon^{1-\theta}}, \frac{\sqrt{6M\epsilon_s}bc}{\epsilon^{1-\theta}}\right\}$  iterations for each stage, and $m = \lceil \log_b(\frac{\epsilon_0}{\epsilon})\rceil$. Then  %there exists $s\in[1,m]$  
\begin{align*}
F(x_m) - F_*\leq 2\epsilon. 
\end{align*}
Hence, the iteration complexity for achieving an $2\epsilon$-optimal solution is $\widetilde O(1/\epsilon^{1-\theta})$. %$ \sum_{s=1}^{\lceil \log_b(\frac{\epsilon_0}{\epsilon})\rceil}\max\left\{ \frac{3D \|A\| bc}{\epsilon^{1-\theta}}, \frac{\sqrt{6M\epsilon_s}bc}{\epsilon^{1-\theta}}\right\}$ in the worst-case.
\end{thm}
\begin{proof}
%We assume that the epoch solutions $x_s \notin \S_{2\epsilon}$ for $s = 1, 2, \dots, m-1$ and induce how many epochs are sufficient for reaching to a point in $2\epsilon$-sublevel set in the worst case. Let $x_{s,\epsilon}^\dag$ denote the closest point to $x_s$ in the $\epsilon$ sublevel set. Since $\S_{\epsilon}\subseteq \S_{2\epsilon}$, therefore $x_s \not\in \S_\epsilon$ and consequentially  $F(x_{s,\epsilon}^\dagger)=F_* + \epsilon$ (using the KKT condition).
Let $x_{s,\epsilon}^\dag$ denote the closest point to $x_s$ in the $\epsilon$ sublevel set and define  $\epsilon_s \triangleq\frac{\epsilon_0}{b^s}$. We will show by induction that $F(x_s) - F_*\leq \epsilon_s +\epsilon$ for $s=0,1,\dots$ which leads to our conclusion when $s=m$. 
The inequality holds obviously for $s=0$. Assuming $F(x_{s-1}) - F_*\leq \epsilon_{s-1} + \epsilon$, we  need to show that $F(x_s) - F_*\leq \epsilon_s +\epsilon$. We apply Corollary~\ref{cor:smoothAPG} to the $s$-th epoch of Algorithm~\ref{alg:HOPS-smooth} and get
\begin{align}\label{eqn:th7-1}
F(x_s) - F(x_{s-1, \epsilon}^\dagger) \leq \frac{\mu_s D^2}{2} + \frac{2\| A \|^2\|x_{s-1, \epsilon}^\dagger - x_{s-1}\|^2}{\mu_s t^2} + \frac{2 M \|x_{s-1, \epsilon}^\dagger - x_{s-1}\|^2}{t^2}
\end{align}

First, we assume $F(x_{s-1}) - F_* \leq \epsilon$, i.e. $x_{s-1} \in \mathcal S_{\epsilon}$. Then we have $x_{s-1, \epsilon}^\dagger = x_{s-1}$ and 
\begin{align*}
F(x_s) - F(x_{s-1, \epsilon}^\dagger) \leq \frac{D^2\mu_s}{2} \leq \frac{\epsilon_s}{3}
\end{align*}
As a result, 
\begin{align*}
F(x_s) - F_* \leq F(x_{s-1, \epsilon}^\dagger) - F_* + \frac{\epsilon_s}{3} \leq \epsilon + \epsilon_s
\end{align*}
Next, we consider $F(x_{s-1}) - F_* > \epsilon$, i.e. $x_{s-1} \notin \S_{\epsilon}$. Then we have $F(x_{s-1, \epsilon}^\dagger)  - F_* = \epsilon$.
Recall that
\begin{align}\label{eqn:th7-2}
\|x_{s-1} - x^\dagger_{s-1,\epsilon}\| \leq \frac{c \epsilon_{s-1}}{\epsilon^{1-\theta}}
\end{align}
%We can get
%\begin{align*}
%F(x_s) - F(x_{s-1, \epsilon}^\dagger) & \leq \frac{\mu_s D^2}{2} + \frac{2\| A \|^2 c^2 \epsilon_{s-1}^2 }{\mu_s \epsilon^{2(1-\theta)} t^2} + \frac{2 M c^2 \epsilon_{s-1}^2}{\epsilon^{2(1-\theta)} t^2}
%\end{align*}
Combining (\ref{eqn:th7-1}) and (\ref{eqn:th7-2}) and using the fact that  $\mu_s = \frac{2\epsilon_s}{3D^2}$ and $t \geq \max\left\{ \frac{3D \|A\| bc}{\epsilon^{1-\theta}}, \frac{\sqrt{6M\epsilon_s}bc}{\epsilon^{1-\theta}}\right\}$, we get
\begin{align*}
F(x_s) - F(x_{s-1, \epsilon}^\dagger) & \leq \frac{\epsilon_s}{3}+ \frac{3D^2\| A \|^2 c^2 \epsilon_{s-1}^2 }{\epsilon_s \epsilon^{2(1-\theta)} t^2} + \frac{2 M c^2 \epsilon_{s-1}^2}{\epsilon^{2(1-\theta)} t^2} \\
& \leq \frac{\epsilon_s}{3}+ \frac{\epsilon_{s-1}^2 }{3\epsilon_s b^2 } + \frac{\epsilon_{s-1}^2 }{3\epsilon_s b^2 } = \epsilon_s
\end{align*}
%\begin{align*}
%F(x_s) - F(x_{s-1, \epsilon}^\dagger) & \leq \frac{\mu_s D^2}{2} + \frac{2\| A \|^2 c^2 \epsilon_{s-1}^2 }{\mu_s \epsilon^{2(1-\theta)} t^2} + \frac{2 M c^2 \epsilon_{s-1}^2}{\epsilon^{2(1-\theta)} t^2}\\ 
%							&+ \left[ \frac{D^2 \|A\|^2 + D\|A\|\sqrt{D^2\|A\|^2+2M\epsilon_0}}{\epsilon_s} + 2 M \right] \frac{\|x_{s-1, \epsilon}^\dagger - x_{s-1}\|^2}{t^2}\\
%							& \leq \frac{D\| A \| \epsilon_s}{D \|A\| + \sqrt{D^2\|A\|^2+2M\epsilon_0}} \\ 
%							&+ \left[ \frac{D^2 \|A\|^2 + D\|A\|\sqrt{D^2\|A\|^2+2M\epsilon_0}}{\epsilon_s} + 2 M \right] \left[\frac{\epsilon_{s}}{D \| A \|+\sqrt{D^2 \| A \|^2+ 2M\epsilon_0}}\right]^2\\
%							& = \frac{2D\| A \| \epsilon_s}{D \|A\| + \sqrt{D^2\|A\|^2+2M\epsilon_0}} + 2 M  \left[\frac{\epsilon_{s}}{D \| A \|+\sqrt{D^2 \| A \|^2+ 2M\epsilon_0}}\right]^2\\
%							& \leq \frac{2D\| A \| \epsilon_s}{D \|A\| + \sqrt{D^2\|A\|^2+2M\epsilon_s}} + 2 M  \left[\frac{\epsilon_{s}}{D \| A \|+\sqrt{D^2 \| A \|^2+ 2M\epsilon_s}}\right]^2\\
%							& = \epsilon_s
%\end{align*}
which together with the fact that $ F(x^\dagger_{s-1,\epsilon})=F_*+\epsilon$ implies
\begin{align*}
F(x_s) - F_*& \leq \epsilon  + \epsilon_s
\end{align*}
Therefore by induction, we have
\[
F(x_m) - F_*\leq \epsilon_m + \epsilon = \frac{\epsilon_0}{b^m} +\epsilon\leq 2\epsilon 
\]
where the last inequality is due to the value of $m=\lceil \log_b(\frac{\epsilon_0}{\epsilon})\rceil$.

In fact, the number of iteration in each stage depends on $s$, then the iteration complexity for achieving an 2$\epsilon$-optimal solution is
\begin{align*}
 \sum_{s=1}^{m} \max\left\{ \frac{3D \|A\| bc}{\epsilon^{1-\theta}}, \frac{\sqrt{6M\epsilon_s}bc}{\epsilon^{1-\theta}}\right\} & \leq \sum_{s=1}^{m} \frac{3D\|A\|bc+\sqrt{6M\epsilon_s}bc}{\epsilon^{1-\theta}} \\
& = \frac{3D\|A\|bc}{\epsilon^{1-\theta}} \left\lceil \log_b\left(\frac{\epsilon_0}{\epsilon}\right) \right\rceil + \sum_{s=1}^{m} \frac{\sqrt{6M\epsilon_0}bc}{\sqrt{b^s}\epsilon^{1-\theta}} \\
& \leq \frac{3D\|A\|bc}{\epsilon^{1-\theta}} \left\lceil \log_b\left(\frac{\epsilon_0}{\epsilon}\right) \right\rceil + \frac{\sqrt{6M\epsilon_0}bc}{(\sqrt{b}-1)\epsilon^{1-\theta}} 
\end{align*}
%If the objective function $F(x)$ satisfies  the local error bound which defined in \textbf{Defination 4}. Then, HOPS can achieve a lower iteration complexity than $\widetilde O(1/\epsilon^{1-\theta})$ with $\theta \in (0,1]$.
\end{proof}

\subsection{HOPS with a $p$-norm}
As we mentioned in the paper, we can generalize the results to a smoothness definition with respect to a $p$-norm $\|x\|_p$ with $p\in(1,2]$, which makes $\frac{1}{2}\|x\|_p^2$ a $(p-1)$-strongly convex function w.r.t $\|\cdot\|_p$. Algorithm~\ref{alg:0} and Algorithm~\ref{alg:APG-smooth} remain the same with $\sigma_1 = p-1$ except that the norm $\|\cdot\|$ is replaced with $\|\cdot\|_p$ in the updates of $x_k$ and $z_k$. In order to have  efficient updates using a $p$-norm, we assume $\Omega_1=\R^d$. Similarly as before, we introduce several notations. Let $dist_p(x,\Omega_*)=\min_{z\in\Omega_*}\|x - z\|_p$. Let $x_\epsilon^\dagger$ denote the closest point in the $\epsilon$-sublevel set to $x$ measured in $p$-norm, i.e, 
\[
x_\epsilon^\dagger = \arg\min_{z\in\R^d}\|z - x\|^2_p, \quad s.t. \quad F(z)\leq F_* + \epsilon
\]

The following lemma is a generalization of Lemma~\ref{lem:1} to a $p$-norm. 
\begin{lemma}[\citep{DBLP:journals/corr/arXiv:1512.03107}]\label{lem:4}
For any $x\in\R^d$ and $\epsilon>0$, we have
\[
\|x - x^\dagger_\epsilon\|_p\leq \frac{dist_p(x^\dagger_\epsilon, \Omega_*)}{\epsilon}(F(x) - F(x_\epsilon^\dagger))
\]
where $x^\dagger_\epsilon\in\mathcal S_\epsilon$ is the closest point in the $\epsilon$-sublevel set to $x$. 
\end{lemma}
The proof of the above  lemma can be also found in~\citep{DBLP:journals/corr/arXiv:1512.03107}. For completeness, we give the proof in Appendix. 

To establish the improved convergence, we assume the following local error bound condition using the $p$-norm.
\begin{definition}[Local error bound]
A function $F(x)$ is said to satisfy a local error bound condition w.r.t a $p$-norm if there exist $\theta\in(0,1]$ and $c>0$ such that for any $x\in\mathcal S_\epsilon$
\begin{align}\label{eqn:leb_dual}
dist_p(x, \Omega_*)\leq c(F(x) - F_*)^{\theta}
\end{align}
\end{definition}
The convergence of APG with a $p$-norm is similar to Corollary~\ref{cor:1} in the paper. 
\begin{cor}\label{cor:7}
For any $x\in\R^d$, by running APG with a $p$-norm, we have
\begin{align}\label{eqn:cs3}
F(x_t)- F(x)\leq  \mu D^2/2 + \frac{2L_\mu\|x - x_0\|_p^2}{t^2} 
\end{align}
%In particular in order to have $F(x_t)\leq F_* + \epsilon$, it suffices to set $\mu \leq  \frac{\epsilon}{D^2}$ and $t \geq \frac{2D\|A\|\|x_0 - x_*\|}{\epsilon}$, where $x_*$ is an optimal solution to~(\ref{eqn:opt}). 
\end{cor}
Finally, we have the similar convergence as Theorem~\ref{thm:GDr} for HOPS except that  $\|A\|$ is defined using the $p$-norm of $x$.

\section{Primal-Dual Homotopy Smoothing (PD-HOPS)}
%Finally, we note that the required  number of iterations per-stage $t$ for finding an $\epsilon$ accurate solution depends on an unknown constant $c$ and sometimes $\theta$. Thus, an inappropriate setting of $t$ may lead to a less accurate solution. In practice, it can be tuned to obtain the fastest convergence. A way to eschew the tuning is to consider a primal-dual homotopy smoothing. Basically, we also apply the homotopy smoothing to the dual problem:
%\begin{align*}
%\max_{u\in\Omega_2} \Phi(u) \triangleq -\phi(u) + \min_{x\in\Omega_1}\langle A^{\top}u, x\rangle  + g(x)
%\end{align*}
%Denote by $\Phi_*$ the optimal value of the above problem. It is easy to see that $\Phi_* = F_*$. By extending the analysis and result to the dual problem, we can obtain that $F(x_m) - \Phi(u_m) \leq 4\epsilon$. Thus, we can use the duality gap $F(x_s) - \Phi(u_s)$ as a certificate to monitor the progress of optimization. Due to the limitation of space, we defer the details into the supplement. 

We note that the required  number of iterations per-stage $t$ for finding an $\epsilon$ accurate solution depends on unknown constant $c$ and sometimes $\theta$. Thus, an inappropriate setting of $t$ may lead to a less accurate solution. To address this issue, we present a primal-dual homotopy smoothing. Basically, we also apply the homotopy smoothing to the dual problem:
\begin{align}\label{eqn:dual-opt}
\max_{u\in\Omega_2} \Phi(u) \triangleq -\phi(u) + \underbrace{\min_{x\in\Omega_1}\langle A^{\top}u, x\rangle  + g(x)}\limits_{\psi(u)}
\end{align}
Denote by $\Phi_*$ the optimal value of the above problem. Under some mild conditions, it is easy to see that $\Phi_* = F_*$. By extending the analysis and result to the dual problem, we can obtain that $F(x_m) - \Phi(u_m) \leq 4\epsilon$. Thus, we can use the duality gap $F(x_s) - \Phi(u_s)$ as a certificate to monitor the progress of optimization. In this section, we present the details of primal-dual HOPS. %we firstly present a HOPS for dual problem with analysis of iteraction complexity. Then, we establish a Primal Dual HOPS, which can avoid setting $t$.

\subsection{Nesterov's Smoothing on the Dual Problem}

We construct a smooth function from $\psi_\eta(u)$ that well approximates $\psi(u)$:
\begin{align*}
\psi_\eta(u) = \min_{x\in\Omega_1}\langle A^{\top}u, x\rangle  + g(x) + \eta \omega(x)
\end{align*}
where $\omega(x)$ is a 1-strongly convex function w.r.t. $x$ in terms of a norm $\| \cdot \|$~\footnote{This could be a general norm.}. Similarly, we know that $\psi_\eta(u)$ is a smooth function of $u$ with respect to an Euclidean norm $\|u\|$ with smoothness parameter $L_\eta = \frac{1}{\eta} \| A\|_+^2$, where $\|A\|_+$ is defined by $\|A\|_+=\max_{\|x\|\leq 1}\max_{\|u\|_+\leq 1}\langle A^\top u, x\rangle$. Denote by 
\[
x_\eta(u) = \arg\min_{x\in\Omega_1}\langle A^{\top}u, x\rangle  + g(x) + \eta \omega(x)
\]
The gradient of $\psi_\eta(u)$ is computed by $\nabla \psi_\eta(u) = Ax_\eta(u)$. 
%Then 
%\begin{align}\label{eqn:Dapprox2}
%\psi_\eta(u) - \eta \tilde{D}^2/2 \leq \psi(u) \leq \psi_\eta(u)
%\end{align}
We can see that when $\eta$ is very small, $\psi_\eta(u)$ gives a good approximation of $\psi(u)$. This motivates us to solve the following composite optimization problem
\begin{align*}
\max_{u\in\Omega_2} \Phi_\eta(u) \triangleq -\phi(u) + \psi_\eta(u) 
\end{align*}
Similar to solving the primal problem, an accelerated proximal gradient method for dual problem can be employed to solve the above problem. We present the details in Algorithm~\ref{DAPG}. We present the convergence results for Algorithm~\ref{DAPG} in the following theorem:
\begin{thm}\label{thm:DAPG}\citep{Nesterov:2005:SMN,citeulike:6604666}
Let $\theta_k  = \frac{2}{k+2}$, $\alpha_k = \frac{2}{k+1}, k\geq 0$ or $\alpha_{k+1} = \theta_{k+1} = \frac{\sqrt{\theta_k^4+4\theta_k^2} - \theta_k^2}{2}, k\geq 0$.  For any $u\in\Omega_2$, we have
\begin{align}
\Phi_\eta(u)  - \Phi_\eta(u_t)\leq \frac{2L_\eta\|u - u_0\|^2 }{t^2}
\end{align}
\end{thm}
%Combining the above convergence result with the relation in (\ref{eqn:Dapprox}), we can establish the iteration complexity of Nesterov's smoothing algorithm for solving the original problem (\ref{eqn:dual-opt}).
%\begin{cor}\label{cor:10}
%For any $u\in\Omega_2$, we have
%\begin{align}\label{eqn:cs2}
%\Phi(u)  - \Phi(u_t)\leq  \eta \hat{D}^2/2 + \frac{2L_\eta\|u - u_0\|_+^2}{t^2} 
%\end{align}
%In particular in order to have $\Phi_*  - \Phi(u_t) \leq \epsilon$, it suffices to set $\eta \leq  \frac{\epsilon}{\hat{D}^2}$ and $t \geq \frac{2\hat{D}\|A\|_+\|u_0 - u_*\|_+}{\epsilon}$, where $u_*$ is an optimal solution to~(\ref{eqn:dual-opt}). 
%\end{cor}

\subsection{HOPS for the Dual Problem}
\begin{algorithm}[t]
\caption{An Accelerated Proximal Gradient Method for solving dual problem (\ref{eqn:dual-opt}): $\text{DAPG}(u_0, t, L_\eta)$} \label{DAPG}
\begin{algorithmic}[1]
\STATE \textbf{Input}: the  number of iterations $t$, the initial solution $u_0$, and the smoothness constant $L_\eta$
\STATE Let $\theta_0=1$, $V_{-1}=0$, $\Gamma_{-1}=0$, $r_0 = u_0$
\STATE Let $\alpha_k$ and $\theta_k$ be two sequences given in Theorem~\ref{thm:DAPG}. 
\FOR{$k=0,\ldots, t-1$}
    \STATE Compute $w_{k} = (1-\theta_k)u_k + \theta_k r_k$
    \STATE Compute $v_k=\nabla \psi_\eta(w_{k}) $, $V_k =V_{k-1} - \frac{v_k}{\alpha_k}$, and $\Gamma_{k}=\Gamma_{k-1} + \frac{1}{\alpha_k}$ 
     \STATE Compute $r_{k+1}=\Pi^{L_\eta/\sigma_2}_{V_k, \Gamma_k\phi}(u_0)$ and $u_{k+1} = \Pi^{L_\eta}_{-v_k, \phi}(w_k)$
   \ENDFOR
\STATE \textbf{Output}:  $u_t$
\end{algorithmic}
\end{algorithm}
\begin{algorithm}[t]
\caption{Homotopy Smoothing (HOPS) for solving dual problem (\ref{eqn:dual-opt}) } \label{alg:DHOPS}
\begin{algorithmic}[1]
\STATE \textbf{Input}: the number of stages $m$ and  the  number of iterations $t$ per-stage, and the initial solution $u_0\in\Omega_2$ and a parameter $b>1$. 
\STATE Let $\eta_1 = \epsilon_0/(b \widetilde{D}^2)$
\FOR{$s=1,\ldots, m$}
    \STATE Let $u_{s} = \text{DAPG}(u_{s-1}, t, L_{\eta_{s}})$ 
    \STATE Update $\eta_{s+1} = \eta_s/b$ 
   \ENDFOR
\STATE \textbf{Output}:  $u_m$
\end{algorithmic}
\end{algorithm}

Similar to primal problem, we can also develop  the HOPS for dual problem, which is presented in Algorithm~\ref{alg:DHOPS}. A convergence can be established similarly by exploring a local error bound condition on $\Phi(u)$.  To present the convergence result, we make the following assumptions, which are similar as the primal problem. 
\begin{ass}\label{ass:2} For a concave maximization problem~(\ref{eqn:dual-opt}), we assume
(i) there exist $u_0\in\Omega_2$ and $\epsilon_0\geq 0$ such that $\max_{u\in\Omega_2}\Phi(u) - \Phi(u_0)\leq \epsilon_0$;
(ii) let $\psi(u) = \min_{x\in\Omega_1}\langle A^{\top}u, x\rangle  + g(x)$, where $g(x)$ is a convex function; 
(iii) There exists a  constant $\widetilde{D}$ such that $\max_{x\in\Omega_1}\omega(x)\leq \widetilde{D}^2/2$. 
%(iv) $\Omega_*$ is a non-empty convex compact set.
\end{ass}
Let $\widetilde\Omega_*$ denote the optimal solution set of~(\ref{eqn:dual-opt}). For any $u\in\Omega_2$, let $u^*$ denote the closest optimal solution in $\widetilde{\Omega}_{*}$ to $u$, i.e., $u^* = \arg\min_{v\in\widetilde{\Omega}_{*}}\|v - u\|^2$. % which is unique because $\widetilde{\Omega}_{*}$ is a non-empty convex compact set~\citep{}.  
We denote by  $\widetilde{\mathcal L}_\epsilon$ the  $\epsilon$-level set of $\Phi(u)$ and  by $\widetilde{\mathcal S}_\epsilon$  the $\epsilon$-sublevel set of $\Phi(u)$, respectively, i.e.,
\begin{align}\label{eqn:xepsilon2}
\widetilde{\mathcal L}_\epsilon &= \{u\in\Omega_2: \Phi(u) = \Phi_* - \epsilon\},\quad \widetilde{\mathcal S}_\epsilon = \{u\in\Omega_2: \Phi(u) \geq \Phi_* - \epsilon\}
\end{align}
A local error bound condition is also imposed. 

\begin{definition}[Local error bound (LEB)] \label{leb2}
A function $\Phi(u)$ is said to satisfy a local error bound condition if there exist $\tilde\theta\in(0,1]$ and $\tilde{c}>0$ such that for any $u\in\widetilde{\mathcal S}_\epsilon$
\begin{align}\label{eqn:leb2}
dist(u, \widetilde{\Omega}_*)\leq \tilde{c}(\Phi_* - \Phi(u))^{\tilde\theta}
\end{align}
\end{definition}

\begin{thm}\label{thm:GDr2}
Suppose Assumption~\ref{ass:2} holds and $\Phi(u)$ obeys the local error bound condition.  Let HOPS for dual problem run with $t=O\left( \frac{2b\tilde{c}\widetilde{D}\|A\|_+}{\epsilon^{1-\tilde\theta}}\right)\geq \frac{2b\tilde{c}\widetilde{D}\|A\|_+}{\epsilon^{1-\tilde\theta}}$ iterations for each stage, and $m = \lceil \log_b(\frac{\epsilon_0}{\epsilon})\rceil$. Then  %there exists $s\in[1,m]$  %or after at most  $K$ epochs, we have 
\begin{align*}
\Phi_* - \Phi(u_m) \leq 2\epsilon. 
\end{align*}
%depending on using either subgradients or stochastic subgradients. 
Hence, the iteration complexity for achieving an $2\epsilon$-optimal solution is $\frac{2b \tilde c \widetilde{D}\|A\|_+}{\epsilon^{1-\tilde\theta}}\lceil \log_b(\frac{\epsilon_0}{\epsilon})\rceil$ in the worst-case.
\end{thm}
The above theorem can be  proved similarly as Theorem~\ref{thm:GDr}. 

\subsection{Primal-Dual HOPS}

\begin{algorithm}[t]
\caption{Primal-Dual Homotopy Smoothing (PD-HOPS) for solving (1)} \label{alg:PDHOPS}
\begin{algorithmic}[1]
\STATE \textbf{Input}: the number of stages $m$,  initial solutions $x_0\in\Omega_1, u_0\in\Omega_2$ and a parameter $b>1$
\STATE Let $\epsilon_1 = \frac{\epsilon_0}{b}$, $\mu_1 = \frac{\epsilon_1}{D^2}$, $\eta_1 = \frac{\epsilon_1}{\widetilde{D}^2}$
\FOR{$s=1,\ldots, m$}
\FOR{$k=0, 1, \ldots,$}
    \STATE Update the sequence of $x_{k+1}$ as in Algorithm~\ref{alg:0} starting from $x_{s-1}$
    \STATE Update the sequence of $u_{k+1}$ as in Algorithm~\ref{DAPG} starting from $u_{s-1}$
    \STATE Check occasionally if $F(x_{k+1})-\Phi(u_{k+1}) \leq 2(\epsilon_s + \epsilon)$; break the loop if it is true
 \ENDFOR
     \STATE Update $x_{s} = x_{k+1}$ and $u_{s} = u_{k+1}$
     \STATE Update $\epsilon_{s+1} = \epsilon_s/b$, $\mu_{s+1} = \mu_s/b$ and $\eta_{s+1} = \eta_s/b$
\ENDFOR
\STATE \textbf{Output}:  $(x_m, u_m)$
\end{algorithmic}
\end{algorithm}

As mentioned before, we can use the duality gap $F(x_s) - \Phi(u_s)$ as a certificate to monitor the progress of optimization to address the problem of detecting the number of iterations per-stage $t$. We describe the details in Algorithm~\ref{alg:PDHOPS}. %It shows that we move to next stage if duality gap is not greater than $\epsilon + \epsilon_s$, thus for stage $s$, we have 
%\begin{align}
%F(x_s) - F_* = F(x_s) - \Phi_* \leq   F(x_s) - \Phi(u_s) \leq \epsilon + \epsilon_s
%\end{align}
Suppose Assumptions \ref{ass:1} and \ref{ass:2} hold, following the analysis as in the proof of Theorem~\ref{thm:GDr}, when the number of iterations in the $s$-th epoch denoted by $t_s$ satisfies  $t_s\geq \max \{\frac{2bcD\|A\|}{\epsilon^{1-\theta}}  , \frac{2b\tilde{c}\widetilde{D}\|A\|_+}{\epsilon^{1-\tilde\theta}} \}$, we can have $F(x_s) - F_* \leq \epsilon + \epsilon_s$ and $\Phi_* - \Phi(u_s) \leq \epsilon + \epsilon_s$, so that
\begin{align}
F(x_s) - \Phi(u_s) \leq 2(\epsilon + \epsilon_s)
\end{align}
Hence, as long as the above inequality holds,  we restart the next stage. Then with at most $m= \lceil \log_b(\epsilon_0/\epsilon)\rceil$ epochs we  have
\begin{align}
F(x_m) - \Phi(u_m) \leq 2(\epsilon + \epsilon_m)\leq 4\epsilon. 
\end{align}
Similarly, we can show that PD-HOPS enjoys an $\widetilde O(\max\{1/\epsilon^{1-\theta}, 1/\epsilon^{1-\tilde\theta}\})$  iteration complexity, where $\tilde \theta$ is the exponent constant in the local error bound of the objective function for dual problem. For example, for linear classification problems with a piecewise linear loss and  $\ell_1$ norm regularizer we can have $\theta =1$ and $\tilde \theta = 1$, and PD-HOPS enjoys a linear convergence.% Due to the limitation of space, we defer the details of PD-HOPS and its analysis into the supplement.  

\section{Experiments}

In this section, we present some experimental results to demonstrate the effectiveness of HOPS and PD-HOPS by comparing with two state-of-the-art algorithms, the first-order Primal-Dual (PD) method~\citep{Chambolle:2011:FPA:1968993.1969036} and Accelerated Proximal Gradient (APG) methods. For APG, we implement two variants, where APG-D refers to the variant with the dual averaging style of update on one sequence of points (i.e., Algorithm~\ref{alg:0}) and APG-F refers to the variant of the FISTA style~\citep{Beck:2009:FIS:1658360.1658364}.  Similarly, we also implement the two variants for HOPS. We conduct experiments for solving three problems: (1) an  $\ell_1$-norm regularized hinge loss for linear classification on the w1a dataset~\footnote{https://www.csie.ntu.edu.tw/$\sim$cjlin/libsvmtools/datasets/}; (2) a total variation based ROF model~\citep{rudin1992nonlinear} for image denoising on the Cameraman picture~\footnote{http://pages.cs.wisc.edu/$\sim$swright/TVdenoising/}; (3) a nuclear norm regularized  absolute error minimization  for low-rank and sparse matrix decomposition on a synthetic data. The three problems are discussed in details below.

\begin{table}
\caption{Comparison of different optimization algorithms by the number of iterations and running time in second (mean $\pm$ standard deviation) for achieving a solution that satisfies $F(x) - F_*\leq \epsilon$.}\label{table1}
\vspace*{-0.1in}
\centering
  \def\sym#1{\ifmmode^{#1}\else\(^{#1}\)\fi}%
 %\begin{scriptsize}
 \begin{tiny}
 %\hspace*{-50pt}
  \begin{tabular}{l*{6}{l}}
    %\toprule
    \hline
    & \multicolumn{2}{c}{Linear Classification}  & \multicolumn{2}{c}{Image Denoising} & \multicolumn{2}{c}{Matrix Decomposition} \\
    %\cmidrule(lr){2-3}\cmidrule(lr){4-5} \cmidrule(lr){6-7}
    & \multicolumn{1}{l}{$\epsilon=10^{-4}$} & \multicolumn{1}{l}{$\epsilon=10^{-5}$} & 
      \multicolumn{1}{l}{$\epsilon=10^{-3}$} & \multicolumn{1}{l}{$\epsilon=10^{-4}$}&
       \multicolumn{1}{l}{$\epsilon=10^{-3}$} & \multicolumn{1}{l}{$\epsilon=10^{-4}$} \\
    %\midrule
     \hline
    PD&  9861 (1.58$\pm$0.02)     &  27215 (4.33$\pm$0.06)    &  8078 (22.01$\pm$0.51)  &  34292 (94.26$\pm$2.67)  &  2523 (4.02$\pm$0.10)    &    3441 (5.65$\pm$0.20)     \\
%\midrule
\hline
    APG-D   & 4918 (2.44$\pm$0.22)    &  28600 (11.19$\pm$0.26)     &   179204 (924.37$\pm$59.67)   & 1726043 (9032.69$\pm$539.01) &   1967 (6.85$\pm$0.08)    &  8622 (30.36$\pm$0.11)  \\
    APG-F  &  3277 (1.33$\pm$0.01)&  19444 (7.69$\pm$0.07)   & 14150 (40.90$\pm$2.28)     & 91380 (272.45$\pm$14.56)  &  1115 (3.76$\pm$0.06)    &  4151 (9.16$\pm$0.10)  \\
   % \midrule
    \hline
    HOPS-D   &  1012 (0.44$\pm$0.02)     &  4101 (1.67$\pm$0.01)      &  3542 (13.77$\pm$0.13)  &  4501 (17.38$\pm$0.10) &  224 (1.36$\pm$0.02)    &  313 (1.51$\pm$0.03)   \\
     HOPS-F   & 1009 (0.46$\pm$0.02)    &  4102 (1.69$\pm$0.04)  &  2206 (6.99$\pm$0.15)  & 3905 (16.52$\pm$0.08) &  230 (0.91$\pm$0.01)    &  312 (1.23$\pm$0.01)   \\
   % \midrule
    \hline
    PD-HOPS   &  846 (0.36$\pm$0.01) &  3370 (1.27$\pm$0.02)      &  2538 (7.97$\pm$0.13)  &  3605 (11.39$\pm$0.10)   &  124 (0.45$\pm$0.01)    &  162 (0.64$\pm$0.01)   \\
   % \bottomrule
   \hline
  \end{tabular}
 %\hspace*{-50pt}
 % \{\centering}
  \end{tiny}
 % \end{scriptsize}
\end{table}

\begin{itemize}
\item{
{\bf{Linear Classification:}} In linear classification problems, the goal is to solve the following optimization problem:
		\begin{align*}
			\min_{x \in \R^d}\quad  \frac{1}{n} \sum_{i=1}^{n}\ell(x^\top a_i,y_i) + \lambda r(x)
		\end{align*}
where $(a_i,y_i), i = 1, 2,\dots,n$ denote pairs of  and label of training data, $\ell(x^\top a_i,y_i) $ is loss function, $r(x)$ is regularizer, and $\lambda$ is regularization parameter. 		
In our experiment, we use the hinge loss (a non-smooth function) $\ell(zy) = \max(0, 1-zy) = \max_{\alpha \in [0,1]} \alpha(1-zy)$ for loss function and the $\ell_1$-norm for regularizer:
		\begin{align}\label{eq:linerclass}
			\min_{x \in \R^d} F(x) \triangleq \frac{1}{n} \sum_{i=1}^{n}\max_{u_i \in [0,1]}u_i(1-y_ia_i^\top x) + \lambda \|x\|_1
		\end{align}
	We first write (\ref{eq:linerclass}) into the following equivalent minimax formulation
		\begin{align}
			\min_{x \in \R^d} \max_{u \in [0,1]^n} u^\top A x + \frac{u^\top \mathbf{1}}{n} + \lambda \|x\|_1
		\end{align}
	where matrix $A = -\frac{1}{n} (y_1a_1, y_2a_2,\dots, y_na_n)^{\top}$ and $\mathbf{1}$ is a vector of all ones. Thus, $f(x) = \max_{u \in [0,1]^n} u^\top A x + \frac{u^\top \mathbf{1}}{n} $ and $g(x) = \lambda \|x\|_1$. To apply Nesterov's smoothing technique, we construct  the following smoothed function
		\begin{align}
			f_{\mu}(x) =  \max_{u \in [0,1]^n} u^\top A x + \frac{u^\top \mathbf{1}}{n} - \frac{\mu}{2} \|u\|_2^2
		\end{align}		
	We construct the experiment on the w1a dataset, which contains $2,477$ training examples and $300$ features. We fix the regularization parameter $\lambda = n^{-1}$.
}

\item{	
{\bf{Image Denoising:}} For total variation (TV) based image denoising problem, we consider the following ROF model:
\begin{align}\label{eq:TVL2_ROF}
  \min_{x} \int_{\Omega} | \nabla x | + \frac{\lambda}{2} \| x - h \|_{2}^{2}   ,
\end{align}
\noindent
where $h$ is the observed noisy image, $\Omega \subset \R^{m \times n}$ is the image domain, $\int_{\Omega} | \nabla x |$ is the TV regularization term, and $\lambda$ is the trade-off parameter between regularization and fidelity.
Following the ROF setting in~\citep{Chambolle:2011:FPA:1968993.1969036}, we obtain the following discrete version:
\begin{align}\label{eq:discrete_ROF}
  \min_{x \in X} F(x)\triangleq\| \nabla x \|_{1} + \frac{\lambda}{2} \| x - h \|_{2}^{2}  .
\end{align}
\noindent
where $X = \R^{mn}$ is a finite dimensional vector space, $\nabla x \in Y$ and $Y = X \times X$.
The discrete gradient operator $\nabla x$ is defined as following that has two components:
$$
  (\nabla x)^{1}_{i,j} = \left\{
    \begin{aligned}
      & x_{i+1,j} - x_{i,j} & \text{if} ~ i < m \\
      & 0                   & \text{if} ~ i = m  \nonumber
    \end{aligned}
    \right.
$$
$$
  (\nabla x)^{2}_{i,j} = \left\{
    \begin{aligned}
      & x_{i,j+1} - x_{i,j} & \text{if} ~ j < n \\
      & 0                   & \text{if} ~ j = n , \nonumber
    \end{aligned}
    \right.
$$
\noindent
and $\| \nabla x \|_{1}$ is defined as
$$
\| \nabla x \|_{1} = \sum_{i,j} | (\nabla x)_{i,j} | = \sum_{i,j} \sqrt{((\nabla x)_{i,j}^1)^2 + ((\nabla x)_{i,j}^2)^2}   .
$$
According to~\citep{Chambolle:2011:FPA:1968993.1969036}, we have the minimax formulation of ROF model as
\begin{align}\label{eq:pd_ROF}
  \min_{x \in X} \max_{u \in\Omega_2} -\langle x, \text{div} u \rangle + \frac{\lambda}{2} \| x - h \|_2^2 %- \delta_{\mathcal{U}}(u)  ,
\end{align}
\noindent
where $\Omega_2=\{u: u\in Y, \|u\|_\infty\leq 1\}$, 
%$$
%\delta_{\mathcal{U}}(u) = \left\{
%    \begin{aligned}
%        & 0        & \text{if} ~ u \in \mathcal{U}  , \\
%        & + \infty & \text{if} ~ u \notin \mathcal{U}  .
%    \end{aligned}
%    \right.
%$$
%$\mathcal{U} = \{ u \in Y: \| u \|_{\infty} \leq 1 \}$ is the domain of $u$, and 
$\| u \|_{\infty} = \max_{i,j} \sqrt{(u_{i,j}^1)^2 + (u_{i,j}^2)^2}$, and $\text{div} u$ is the discrete divergence operator~\citep{Chambolle:2011:FPA:1968993.1969036}.  Thus, $f(x) = \max_{u \in \Omega_2} -\langle x, \text{div} u \rangle $ and $g(x) = \frac{\lambda}{2} \| x - h \|_2^2$.
By using Nesterov's smoothing technique, we have the following smoothed function
\begin{align}\label{eq:smoothed_ROF}
  \max_{u \in \Omega_2} -\langle x, \text{div} u \rangle - \frac{\mu}{2}\| u \|_2^2  .
\end{align}
In our experiment, we use Cameraman picture of size 256 $\times$ 256 with additive zero mean Gaussian noise with standard deviation $\sigma = 0.05$ and we set $\lambda = 20$.
}

\item{
{\bf{Matrix Decomposition:}} In low-rank and sparse matrix decomposition problem, suppose given a data matrix $O \in \R^{m \times n}$, we aim to decompose it as 
		\begin{align*}
			O = X + E
		\end{align*}
	where $X \in \R^{m \times n}$ is a  low-rank matrix, and $E \in \R^{m \times n}$ represents errors and it is sparse. 
	We use nuclear norm regularized absolute error minimization:
		\begin{align*}
			\min_{X \in \R^{m \times n}} & F(X) = \|X\|_* + \lambda \|E\|_1 \\ 
			&\text{s.t. } O = X + E
		\end{align*}
where $\|X\|_* = \sum_i \sigma_i (X)$ denotes the nuclear norm of matrix $X$, i.e.,  the summation of singular values  of matrix $X$, and $\|E\|_1 = \sum_{ij} |E_{ij}|$ denotes the $\ell_1$-norm of $E$. The above formulation is equavilent to
		\begin{align}\label{eq:matrixDec}
			\min_{X \in \R^{m \times n}} F(X) = \|X\|_* + \lambda \|O - X\|_1
		\end{align}		
We first write (\ref{eq:matrixDec}) into the following equivalent minimax formulation
		\begin{align}
			\min_{X \in \R^{m \times n}} \max_{\|U\|_{\infty} \leq 1}  -\lambda \langle X, U\rangle + \lambda \langle O, U\rangle + \|X\|_*
		\end{align}
	where $U \in \R^{m \times n}$ and $\|U\|_{\infty} = \max_{ij}|U_{ij}|$. Thus, $f(X) =  \max_{\|U\|_{\infty} \leq 1}  -\lambda \langle X, U\rangle + \lambda \langle O, U\rangle $ and $g(X) = \|X\|_*$. To apply Nesterov's smoothing technique, we consider the following smoothed function
		\begin{align}
			f_{\mu}(X) = \max_{\|U\|_{\infty} \leq 1}  -\lambda \langle X, U\rangle + \lambda \langle M, U\rangle - \frac{\mu}{2} \|U\|_F^2
		\end{align}		
	We set the regularization parameter $\lambda = (\max\{m,n\})^{-0.5}$. We conduct experiment on  a synthetic data with $m=n=100$. To generate the corrupted matrix $O \in \R^{m \times n}$, we first obtain two orthogonal matrices $S_1 \in \R^{m \times k}$ and $S_2 \in \R^{n \times k}$ ($k=10$) by Gaussian distribution. The low rank matrix $X$ can be calculated by $X = S_1 S_2^{\top}$. Then we randomly add Gaussian noise to $10\%$ elements of $X$ and obtain the corrupted matrix $O$.	
}
\end{itemize}

To make fair comparison, we stop each algorithm when the optimality gap is less than a given $\epsilon$ and count the number of  iterations  and the running time that each algorithm requires.  The optimal value is obtained by running PD with a sufficiently large number of iterations such that the duality gap is very small.  {We repeat each algorithm 10 times for solving a particular problem and then report the averaged running time in second and the corresponding standard deviations. The running time of PD-HOPS only accounts the time for updating the primal variable since the updates for the dual variable are fully decoupled from the primal updates and can be carried out in parallel.} 
For APG, we use the backtracking trick to tune $L_\mu$. For HOPS, we tune the number of iterations $t$ in each epoch among several values in the range of $\{10, 50, 100, 150, 200, 250, 300, 350, 400, 500, 1000\}$ and the parameter $b$ among $\{1.2, 2, 2.5, 3, 3.5, 4, 5, 10, 25\}$, and report the best results. We also tune the values of parameters $\sigma$ and $\tau$ and report the best results for PD. We present the comparison of different algorithms on different tasks in Table~\ref{table1}, where for PD-HOPS we only report the results of using the faster variant of APG, i.e.,  APG-F.  From the results, we can see that (i) HOPS converges consistently faster than their APG variants especially when $\epsilon$ is small; (ii) PD-HOPS allows for choosing the number of iterations at each epoch automatically, yielding faster convergence speed than  HOPS with manual  tuning; (iii) both HOPS and PD-HOPS are significantly faster than PD.

 %More details about the formulations and experimental setup can be found in the supplement.%The algorithms are implemented by Matlab and the experiments are conducted on a server with Intel Xeon 2.70GHZ CPU. 
%To make fair comparison, we stop each algorithm when the optimality gap is less than a given $\epsilon$ and count the number of  iterations  and the running time that each algorithm requires. %Normally, HOPS-PD is more effective than HOPS. For linear classification probelem, if $\epsilon$ is large, PD is slight worse than HOPS and HOPS-PD, however, while $\epsilon$ is small, HOPS and HOPS-PD are much more effective. For image denoising problem, HOPS and HOPS-PD significantly outperform other compared algorithms.
%For PD, we implement both the first and second algorithms in~\citep{Chambolle:2011:FPA:1968993.1969036}, which refer to PD1 and PD2, respectively.

\section{Conclusions}
\vspace*{-0.1in}
In this paper, we have developed  a homotopy smoothing (HOPS) algorithm for solving a family of structured non-smooth optimization problems with formal guarantee on the iteration complexities. We show that the proposed  HOPS can achieve a lower iteration complexity of $\widetilde O(1/\epsilon^{1-\theta})$ with $\theta\in(0,1]$ for obtaining an $\epsilon$-optimal solution under a mild local error bound condition. The experimental results on three different tasks  demonstrate the effectiveness of HOPS.

\section*{Acknowlegements}
Y. Xu and T. Yang are partially supported by National Science Foundation (IIS-1463988, IIS-1545995).
%We thank the anonymous reviewers for their helpful comments.  Y. Xu and T. Yang are partially supported by National Science Foundation (IIS-1463988, IIS-1545995).

\bibliography{all}

\appendix
\section{Proof of Lemma~\ref{lem:1}}
The lemma is an immediate result from~\citep{DBLP:journals/corr/arXiv:1512.03107}.  For completeness, we give the proof here.
\begin{proof}
Consider $\|x\|$ to be an Euclidean norm.  We first recall the definition of $x^\dagger_\epsilon$: 
\begin{equation}\label{eqn:ec}
\begin{aligned}
x_\epsilon^{\dagger}&=\arg\min_{z\in\mathcal S_\epsilon}\|z - x\|^2%_2,\quad \text{s.t.}\quad F(z)\leq F_* + \epsilon.
\end{aligned}
\end{equation}
where  $\mathcal S_\epsilon=\{x\in\Omega_1: F(x)\leq F_* +\epsilon\}$ is the sublevel set. 
We assume $x\not\in\mathcal S_\epsilon$, otherwise the conclusion holds trivially. Thus $F(x_\epsilon^\dagger)=F_*+\epsilon$. %For any $\lambda>0$, let $\Pc(\w^\dagger_\epsilon, \lambda)$ denote the $\lambda$-gradient mapping of $\w_\epsilon^\dagger$. 
By the first-order optimality conditions of~(\ref{eqn:ec}), we have for any $z\in\Omega_1$, there exists $\zeta\geq 0$ (the Lagrangian multiplier of problem~(\ref{eqn:ec}))
\begin{equation}\label{eqn:o2}
\begin{aligned}
&(x^\dagger_\epsilon - x + \zeta \partial F(x^\dagger_\epsilon))^{\top}(z - x_\epsilon^\dagger)\geq 0\\
%&( \zeta \partial f(\x_\epsilon^\dagger) +  \Pc(\x_\epsilon^\dagger, \zeta) - \x_\epsilon^\dagger)^{\top}(\u - \Pc(\x^\dagger_\epsilon,\zeta))\geq 0
\end{aligned}
\end{equation}
Let $z = x$  we have
\[
 \zeta \partial F(x^\dagger_\epsilon)^{\top}(x - x^\dagger_\epsilon)\geq \|x - x_\epsilon^\dagger\|^2
\]
We argue that $\zeta>0$, otherwise $x = x_\epsilon^\dagger$ contradicting to the assumption $x\not\in\mathcal S_\epsilon$. 
Therefore
\begin{align}\label{eqn:b1}
F(x) - F(x^\dagger_\epsilon)\geq  \partial F(x^\dagger_\epsilon)^{\top}(x - x^\dagger_\epsilon)\geq \frac{\|x - x_\epsilon^\dagger\|^2}{\zeta}=  \frac{\|x - x_\epsilon^\dagger\|}{\zeta}\|x - x_\epsilon^\dagger\|
\end{align}
Next we prove that $\zeta$ is upper bounded. Since
\begin{align*}
-\epsilon = F(x^*_\epsilon) - F(x^\dagger_\epsilon)\geq (x_\epsilon^* - x^\dagger_\epsilon)^{\top}\partial F(x_\epsilon^\dagger)
\end{align*}
where $x^*_\epsilon$ is the closest point to $x^\dagger_\epsilon$ in the optimal set. Let $z=x^*_\epsilon$ in the inequality of~(\ref{eqn:o2}), we have
\begin{align*}
(x_\epsilon^\dagger - x)^{\top}(x^*_\epsilon - x^\dagger_\epsilon)\geq \zeta ( x^\dagger_\epsilon - x^*_\epsilon)^{\top}\partial F(x^\dagger_\epsilon)\geq \zeta \epsilon
\end{align*}
Thus 
\[
\zeta \leq \frac{(x_\epsilon^\dagger - x)^{\top}(x^*_\epsilon - x^\dagger_\epsilon)}{\epsilon}\leq \frac{dist(x_\epsilon^\dagger, \Omega_*)\|x_\epsilon^\dagger - x\|}{\epsilon}
\]
Therefore 
\[
\frac{\|x - x_\epsilon^\dagger\|}{\zeta}\geq \frac{\epsilon}{dist(x_\epsilon^\dagger, \Omega_*)}
\]
Combining the above inequality with~(\ref{eqn:b1}) we have
\[
\|x - x_\epsilon^\dagger\|\leq \frac{dist(x_\epsilon^\dagger, \Omega_*)}{\epsilon}(F(x) - F(x^\dagger_\epsilon))
\]
which completes the proof. 
\end{proof}

\section{Proof of Lemma~\ref{lem:4}}
The proof of Lemma~\ref{lem:4} can be also found in~\citep{DBLP:journals/corr/arXiv:1512.03107}.  For completeness, we give the proof here.
\begin{proof}
We assume that $x\not\in\mathcal S_\epsilon$, otherwise  $x = x_\epsilon^\dagger$ and the lemma holds trivially. Thus $x^\dagger_\epsilon\in\mathcal L_\epsilon$, i.e., $F(x^\dagger_\epsilon)=F_* + \epsilon$. Note that
\[
\frac{\partial \|x\|_p}{\partial x_i}  = \frac{|x_i|^{p-1}sign(x_i)}{\|x\|_p^{p-1}}, \quad (\nabla\frac{1}{2}\|x\|_p^2)_i = \|x\|_p^{2-p}|x_i|^{p-1}sign(x_i)
\]

By the definition of $x_\epsilon^\dagger$ and the Lagrangian  theory, there exists a Lagrangian multiplier  $\zeta\geq 0$ and a subgradient $v^\dagger_\epsilon\in\partial F(x^\dagger_\epsilon)$ such that
\[
 \|x^\dagger_\epsilon - x\|^{2-p}_p|[x^\dagger_{\epsilon}-x]_i|^{p-1}\text{sign}([x^\dagger_{\epsilon}-x]_i)+ \zeta [v^\dagger_\epsilon]_i= 0, \forall i.
\]
It is clear that $\zeta>0$, otherwise  $x=x^\dagger_\epsilon$ that contradicts to $x\not\in\mathcal S_\epsilon$. 
By the convexity of $F(\cdot)$ we have
\begin{align*}
F(x) - F(x^\dagger_\epsilon)&\geq (x - x^\dagger_\epsilon)^{\top}v^\dagger_\epsilon =\frac{1}{\zeta} \|x^\dagger_\epsilon - x\|^{2-p}_p\sum_{i=1}^d|[x^\dagger_{\epsilon}-x]_i|^{p-1}\text{sign}([x^\dagger_{\epsilon}-x]_i)[x^\dagger_{\epsilon}-x]_i\\
&=\frac{1}{\zeta}\|x^\dagger_\epsilon - x\|^{2-p}_p\|x^\dagger_\epsilon- x\|_p^p = \frac{\|x^\dagger_\epsilon - x\|^{2}_p}{\zeta}\end{align*}
Next, we bound $\zeta$. Let $1/p + 1/q = 1$.
Since 
\begin{equation}
\begin{split}
\zeta^q\sum_{i}|[ v^\dagger_\epsilon]_i|^q &=  \|x^\dagger_\epsilon - x\|_p^{q(2-p)}\sum_{i=1}^d|[x^\dagger_{\epsilon}-x]_i|^{q(p-1)} =  \|x^\dagger_\epsilon - x\|_p^{q(2-p)}\sum_{i=1}^d|[x^\dagger_{\epsilon}-x]_i|^{p} \\
&=  \|x^\dagger_\epsilon - x\|_p^{q(2-p)+p} = \|x^\dagger_\epsilon - x\|_p^{p/(p-1)}
\end{split}
\end{equation}
Thus
\[
\frac{1}{\zeta}\geq \frac{\|v^\dagger_\epsilon\|_q}{\|x^\dagger_\epsilon - x\|_p^{p/(q(p-1))}} 
\]
To lower bound $\|v^\dagger_\epsilon\|_q$, we explore the convexity of $F(x)$. By the convexity of $F(\cdot)$, 
\[
F(x_\epsilon^*) - F(x^\dagger_\epsilon)\geq (x^*_\epsilon - x^\dagger_\epsilon)^{\top}v^\dagger_\epsilon
\]
where $x^*_\epsilon$ is the closest point in $\Omega_*$ to $x^\dagger_\epsilon$. Then we have
\begin{align*}
 \|x^*_\epsilon - x^\dagger_\epsilon\|_p&\|v^\dagger_\epsilon\|_q\geq  -(x^*_\epsilon - x^\dagger_\epsilon)^{\top}v^\dagger_\epsilon\geq F(x^\dagger_\epsilon) - F(x_\epsilon^*) = F(x^\dagger_\epsilon) - F_* = \epsilon
 \end{align*}
Then
\[
\frac{1}{\zeta}\geq \frac{\epsilon}{\|x^\dagger_\epsilon - x\|_p^{p/(q(p-1))} \|x^*_\epsilon - x^\dagger_\epsilon\|_p}  = \frac{\epsilon}{\|x^\dagger_\epsilon - x\|_p \|x^*_\epsilon - x^\dagger_\epsilon\|_p} 
\]
Therefore 
\[
F(x) - F(x^\dagger_\epsilon)\geq  \frac{\|x^\dagger_\epsilon - x\|^{2}_p}{\zeta}\geq \frac{\epsilon\|x^\dagger_\epsilon - x\|^{2}_p}{\|x^\dagger_\epsilon - x\|_p \|x^*_\epsilon - x^\dagger_\epsilon\|_p} =  \frac{\epsilon\|x^\dagger_\epsilon - x\|_p}{ \|x^*_\epsilon - x^\dagger_\epsilon\|_p}
\]
i.e., 
\[
\|x^\dagger_\epsilon - x\|_p\leq \frac{ \|x^*_\epsilon - x^\dagger_\epsilon\|_p}{\epsilon}(F(x) - F(x^\dagger_\epsilon)) =\frac{ dist_p( x^\dagger_\epsilon, \Omega_*)}{\epsilon}(F(x) - F(x^\dagger_\epsilon))
\]
\end{proof}

\end{document}